\documentclass[a4paper,11pt]{amsart}
\usepackage[OT2,OT1]{fontenc}
\usepackage[latin1]{inputenc}\usepackage{amsmath}
\usepackage{amsthm}
\usepackage{amscd}
\usepackage{amssymb}
\usepackage{pb-diagram, pb-xy}
\usepackage[all]{xy}
\usepackage{hyperref}
\hypersetup{colorlinks=true} 
\usepackage{dsfont}

\newtheorem{thm}{Theorem}[section]
\newtheorem{cor}[thm]{Corollary}
\newtheorem{lem}[thm]{Lemma}
\newtheorem{prop}[thm]{Proposition}

\theoremstyle{definition}
\newtheorem{defn}[thm]{Definition}
\newtheorem{rem}[thm]{Remark}

\newtheorem{para}[thm]{--}

\newcommand{\angl}[1]{\left\langle #1\right\rangle}

\newcommand{\tq}{\: | \:}

\newcommand{\qqet}{\qquad\mbox{and}\qquad}

\DeclareMathOperator{\ad}{ad}
\DeclareMathOperator{\cl}{cl}
\DeclareMathOperator{\Der}{Der}%
\DeclareMathOperator{\End}{End}
\DeclareMathOperator{\Ext}{Ext}
\DeclareMathOperator{\Gal}{Gal}
\DeclareMathOperator{\GL}{GL}
\DeclareMathOperator{\gr}{gr}
\DeclareMathOperator{\Hdg}{Hdg}%
\DeclareMathOperator{\Hom}{Hom}
\DeclareMathOperator{\id}{id} %
\DeclareMathOperator{\im}{im} %
\DeclareMathOperator{\Lie}{Lie} %
\DeclareMathOperator{\Mor}{Mor} %
\DeclareMathOperator{\res}{res}
\DeclareMathOperator{\spec}{spec}

\newcommand{\hotimes}{\:\widehat\otimes\:}

\newcommand{\Tell}{\textup{T}_\ell}
\newcommand{\Tnot}{\textup{T}_0}

\newcommand{\Vell}{\textup{V}_{\!\ell}}
\newcommand{\Vnot}{\textup{V}_{\!0}}
\newcommand{\VdR}{\textup{V}_{\!\textup{dR}}}
\newcommand{\VIA}{\textup{V}_{\!\mathbb{A}}}

\newcommand{\updR}{\textup{dR}}%

\newcommand{\upmot}{\textup{mot}}
\newcommand{\upet}{\textup{\'et}}

\newcommand{\IA}{\mathbb{A}}
\newcommand{\IC}{\mathbb{C}}

\newcommand{\IG}{\mathbb{G}}
\newcommand{\IL}{\mathbb{L}}

\newcommand{\IQ}{\mathbb{Q}}
\newcommand{\IR}{\mathbb{R}}

\newcommand{\IZ}{\mathbb{Z}}

\newcommand{\cM}{\mathcal M}
\newcommand{\cO}{\mathcal O}

\newcommand{\cT}{\mathcal T}

\newcommand{\cHom}{\mathcal Hom}
\newcommand{\cEnd}{\mathcal End}
\newcommand{\cExt}{\mathcal Ext}

\newcommand{\fl}{\mathfrak l}
\newcommand{\fg}{\mathfrak g}

\newcommand{\fh}{\mathfrak h}

\newcommand{\fu}{\mathfrak u}

\renewcommand{\to}{\xrightarrow{\quad}}
\renewcommand{\mapsto}{\longmapsto}

\date{\today}

\title{On the Mumford--Tate conjecture for 1--motives}%
\author{Peter Jossen}%

\begin{document}

\begin{abstract}
We show that the statement analogous to the Mumford--Tate conjecture for abelian
varieties holds for 1--motives on unipotent parts. This is done by comparing the
unipotent part of the associated Hodge group and the unipotent part of the image
of the absolute Galois group with the unipotent part of the motivic fundamental
group.
\end{abstract}

\maketitle%

\vspace{7mm}%
\tableofcontents

\vspace{14mm}
\section*{Introduction and overview}

\begin{par}
Let $k$ be a field which is finitely generated over $\IQ$, with algebraic
closure $\overline k$. Let $X$ be a separated scheme of finite type over $k$,
and let $i\geq 0$ be an integer. For every embedding $\sigma:k\to \IC$ the
cohomology group
$$V_0 = H^i(X(\IC),\IQ)$$
carries a mixed rational Hodge structure. The fundamental group of the Tannakian
subcategory of the category of mixed Hodge structures generated by $V_0$ is
called the \emph{Mumford--Tate group} of $V_0$. It is an algebraic subgroup of
$\GL_{V_0}$, which is reductive in the case $X$ is smooth and proper. For any
prime number $\ell$, the $\ell$--adic \'etale cohomology group
$$V_\ell = H^i_\upet(X_{\overline k},\IQ_\ell)$$
is a Galois representation, conjectured to be semisimple if $X$ is smooth and
proper. The vector spaces $V_0$ and $V_\ell$ both carry a \emph{weight
filtration}, and, once an extension of $\sigma$ to $\overline k$ is chosen,
there is a canonical, natural isomorphism of filtered $\IQ_\ell$--vector spaces
$V_0\otimes \IQ_\ell \cong V_\ell$ called \emph{comparison isomorphism}. The
general Mumford--Tate conjecture states that the image of the Galois group
$\Gal(\overline k|k)$ in the group of $\IQ_\ell$--linear automorphisms of
$V_\ell$ contains an open subgroup which is contained and open in the
$\IQ_\ell$--points of the Mumford--Tate group associated with the Hodge
structure $V_0$, via the comparison isomorphism. The classical Mumford--Tate
conjecture is the special case where $X$ is an abelian variety and $i=1$.
\end{par}

\begin{par}
Although a conjecture in general, the classical Mumford--Tate conjecture is known to be true in a variety of
cases, see \cite{Ribe90} or the introduction of \cite{Vasiu} for overviews. For abelian varieties of complex
multiplication type, the statement of the conjecture follows from Faltings's theorems, but was proven
already in 1968 by Pohlmann \cite{Pohlmann}. Serre proved it for elliptic curves in \cite{SerreEllAdic}, and
for abelian varieties $A$ with $\End_{\overline k}A = \IZ$ of dimension $2$, $4$, $6$ or an odd number in
\cite{SerreCours}. Serre's results were improved by Pink in \cite{Pink98}. Recent progress on the question
is due to Vasiu who shows in \cite{Vasiu} the statement of the conjecture to be true for an abelian variety
$A$ under some conditions on the Shimura pair associated with $H_1(A(\IC),\IQ)$.
\end{par}

\begin{par}
The general Mumford--Tate conjecture fits well into the framework of motives. We will show that it holds for 1--motives, provided the classical
Mumford--Tate conjecture holds for abelian parts. Recall from \cite{Deli74} that a 1--motive $M$ over $k$ is given by a diagram of commutative group schemes over
$k$ of the form
$$M = \left[\begin{diagram} \setlength{\dgARROWLENGTH}{4mm}
\node[3]{Y}\arrow{s,l}{u}\\
\node{0}\arrow{e}\node{T}\arrow{e}\node{G}\arrow{e}\node{A}\arrow{e}\node{0}
\end{diagram}\right]$$
where $G$ is an extension of an abelian variety $A$ by a torus $T$, and $Y$ is \'etale locally constant, locally isomorphic to a finitely generated free $\IZ$--module. In other words, $Y$ is a Galois--module which is finitely generated and free as a commutative group. We can look at tori, abelian varieties and finitely generated free groups with Galois action as 1--motives, and 1--motives come equipped with a weight filtration $W$ such that
$$\gr^W_0(M) = Y \qquad\qquad \gr^W_{-1}(M) = A\qquad\qquad \gr^W_{-2}(M) = T$$
With every 1--motive $M$ are associated $\ell$--adic Galois representations
$\Vell M$ and having chosen a complex embedding $\sigma: k\to \IC$ also a mixed
Hodge structure $\Vnot M$. The choice of an extension of $\sigma$ to an
embedding $\overline k\to \IC$ yields a natural comparison isomorphism $\Vnot M
\otimes\IQ_\ell \cong \Vell M$ which is compatible with the weight filtration.
We write $\fl^M$ for the Lie algebra associated with the image of
$\Gal(\overline k|k)$ in $\GL(\Vell M)$ and $\fh^M$ for the Lie algebra of the
Mumford--Tate group of $\Vnot M$. The Lie algebras $\fl^M \subseteq
\End_{\IQ_\ell}(\Vell M)$ and $\fh^M \subseteq \End_{\IQ}(\Vnot M)$ both carry a
two step filtration induced by the weight filtration on $\Vell M$ and $\Vnot M$
respectively which we also denote by $W$:
$$0 \subseteq W_{-2}\fl^M\subseteq W_{-1}\fl^M \subseteq \fl^M \qqet 0 \subseteq
W_{-2}\fh^M\subseteq W_{-1}\fh^M \subseteq \fh^M$$
The Lie algebras $W_{-1}\fl^M$ and $W_{-1}\fh^M$ are the nilpotent
radicals of $\fl^M$ and $\fh^M$ respectively, the reductive Lie algebras
$\gr_0^W(\fl^M)$ and $\gr_0^W(\fh^M)$ are the ones classically associated with
the abelian variety $A = \gr^W_{-1}(M)$, if $A\neq 0$. The comparison isomorphism
permits us to identify $\fh^M\otimes\IQ_\ell$ with a Lie subalgebra of
$\End_{\IQ_\ell}(\Vell M)$. It is a theorem of Brylinski (\cite{Brylinski}, our
theorem \ref{Thm:BrylinskiDeligne}) that this subalgebra is independent of the
chosen complex embeddings. With this identification made, we can state our first
main result:
\end{par}

\vspace{4mm}
\begin{par}{\bf Theorem 1.}\emph{
Let $M$ be a 1--motive over a finitely generated subfield $k$ of $\IC$. The Lie
algebra $\fl^M$ is contained in $\fh^M\otimes\IQ_\ell$, and the equality
$W_{-1}\fl^M = W_{-1}\fh^M\otimes\IQ_\ell$ holds. In particular, the
Mumford--Tate conjecture holds for $M$ if and only if it holds for the abelian
variety $\gr^W_{-1}(M)$.}
\end{par}

\vspace{4mm}
\begin{par}
With every variety $X$ over $k$ one can naturally associate a 1--motive $M^1(X)$
over $k$ such that there are canonical isomorphisms
$$\Vnot M^1(X) \cong H^1(X(\IC),\IQ)(1) \qqet \Vell M^1(X) \cong
H^1_\upet(X_{\overline k},\IQ_\ell)(1)$$
of Hodge structures and Galois representations respectively. For curves this is
a classical construction due to Deligne, in general it is due to Barbieri--Viale
and Srinivas \cite{BaViSri01}. Our first theorem immediately yields:
\end{par}

\vspace{4mm}
\begin{par}{\bf Corollary.}\emph{
Let $X$ be a variety over $k$. The Mumford--Tate conjecture holds for cohomology
in degree $1$ of $X$ if and only if the classical Mumford--Tate conjecture holds
for the albanese of a smooth projective variety birational to $X$.}
\end{par}

\vspace{4mm}
\begin{par}
It is only natural to ask for an analogue of Theorem 1 in positive
characteristic, replacing $k$ by a field which is finitely generated over a
finite field. Alas, there is no Mumford--Tate group in characteristic $p>0$.
However, if we concentrate on the weight $(-1)$--parts, i.e.\ nilpotent
radicals, we can do even better by constructing a motive with which we can
compare $W_{-1}\fl^M$ and $W_{-1}\fh^M$. This motive will be a
semiabelian variety, and was already constructed, following Deligne, by Bertolin
in \cite{Bert03}, where it is called \emph{Lie algebra of the unipotent motivic
fundamental group of $M$}. Our second main result is the following theorem.
\end{par}

\vspace{4mm}
\begin{par}{\bf Theorem 2.}\emph{
With every 1--motive $M$ over a field $k$ is canonically associated a semiabelian variety $P(M)$ over $k$, having the following properties:
\begin{enumerate}
 \item For every field extension $k'|k$, there is a natural isomorphism $P(M)\times_kk' \cong P(M\times_kk')$.
 \item If $k = \IC$, there is a canonical isomorphism of mixed Hodge structures $\Vnot P(M) \cong W_{-1}\fh^M$, where $\fh^M$ denotes the Lie algebra of the
Mumford--Tate group of $\Vnot M$.
 \item Let $\overline k$ be an algebraic closure of $k$ and let $\ell$ be a prime number not equal to the characteristic of $k$. Then, $\Vell P(M)$ can be canonically identified with a $\Gal(\overline k|k)$--subrepresentation of $\End(\Vell M)$. If $k$ is finitely generated over its prime field, then the equality $\Vell P(M)=W_{-1}\fl^M$ holds, where $\fl^M$ denotes the Lie algebra of the image of $\Gal(\overline k|k)$ in $\GL(\Vell M)$.
\end{enumerate}}
\end{par}

\vspace{4mm}
\begin{par}
We willl also formulate an adelic refinement of part (3) (Theorem \ref{Thm:ComparisonEll}). The proof of this refinement is conditional in $k$ has positive characteristic, since it depends on a Galois property of abelian varieties over $k$ which is, to the best of my knowledge, only proven in characteristic $0$. To get an idea of what $P(M)$ and the isomorphisms in the theorem look like, consider a 1--motive $M$ over a field $k$, where $Y=\IZ$ and $T=0$, so $M$ is given by an abelian variety $A$ over $k$ and a rational point $a = u(1)\in A(k)$. In that case, $P(M)$ is defined to be the smallest abelian subvariety of $A$ which contains a multiple of $a$. For
instance, $P(M)=0$ if and only if $a$ is torsion, which is always the case if $k$ is finite. For a fixed prime number $\ell$ and an integer $i\geq 0$, consider the fields
$$k(A[\ell^i]) \qqet k(\ell^{-i}a) $$
obtained by adjoining to $k$ the $\ell^i$--torsion points of $A(\overline k)$,
respectively all $\ell^i$--division points of $a$ in $A(\overline k)$. So
$k(\ell^{-i}a)$ is a Galois extension of $k(A[\ell^i])$, and there is a natural
map 
$$\vartheta: \Gal\big(k(\ell^{-i}a)|k(A[\ell^i])\big) \to A[\ell^i]$$
sending $\sigma$ to $\sigma(b) - b$ where $b\in A(\overline k)$ is any point
such that $\ell^ib=a$. Results of Ribet (\cite{Ribe76,Ribe79}, see also \cite{Hind88},
Appendix 2, Lemme I,bis) state that if $k$ is a number field, the image of the
map $\vartheta$ is contained in the subgroup $P(M)[\ell^i]$ of $A[\ell^i]$ with
finite index bounded independently of $i$, and even equal to $P(M)[\ell^i]$ for
all but finitely many $\ell$. Passing to limits over $i$ and then passing to Lie
algebras gives the isomorphism claimed in part (3) of our theorem. Let it be
acknowledged that Hindry's reformulation of Ribet's result was seminal to our
general construction. 
\end{par}

\vspace{4mm}
\begin{par}
An important application of 1--motives is their use as a tool in the study of the group rational points $G(k)$ of an abelian or semiabelian variety $G$ over a field $k$. This is no surprise, since to give a $k$--rational point on $G$ is the same as to give a morphism $\IZ\to G$ over $k$. For instance, if $k$ is a number field, direct consequences of our theorems in the case where $M = [Y \to A]$ for an abelian variety $A$ over $k$ play an important role in the proof of local--global principles for subgroups of $A(k)$, as I have shown in \cite{Joss11}. I will give a further illustration concerning \emph{deficient points} on semiabelian varieties over number fields, which were introduced by Jacquinot and Ribet in \cite{Jacq87}. In the case where $k$ is the function field of a complex curve, such points have been studied recently by Bertrand in \cite{Bert11} in connection with a relative version of the Manin--Mumford conjecture.
\end{par}

\vspace{4mm}
\begin{par}{\bf Overview.}
Section \ref{Sec:RealizOf1Mot} is to rehearse 1--motives and related
constructions. In Section \ref{Sec:GaloisInMT} we show that the image of the
absolute Galois group is contained in the $\IQ_\ell$--points of the
Mumford--Tate group. This is essentially a reformulation of a result of Deligne
and Brylinski. In Section \ref{Sec:UnipoMotFG} we construct the semiabelian
variety $P(M)$, that is, the Lie algebra of the unipotent part of the motivic
fundamental group of a 1--motive. We then compare the Lie algebra of the
unipotent motivic fundamental group with the Mumford--Tate group and with the
image of Galois in sections \ref{Sec:CompMotivicHodge} and
\ref{Sec:CompMotivicGalois} respectively, by showing that the Hodge,
respectively the $\ell$--adic realisation of $P(M)$ is canonically isomorphic to
the Lie algebra of the unipotent part of the Mumford--Tate group, respectively
to the Lie algebra of the unipotent part of the image of $\Gal(\overline k|k)$
in $\GL(\Vell M)$. With this we have proven the essential part of our Main
Theorem. However we now have two isomorphisms between the nilpotent radicals of
$\fl^M$ and $\fh^M\otimes\IQ_\ell$, the one given in the Main Theorem, the other
via comparison with the motivic fundamental group. We will check in section
\ref{Sec:Compatibility} that they are the same, and deduce our main theorems as
stated above. In section \ref{Sec:Corollaries} we give some corollaries to our
main theorems, concerning deficient points. The appendix contains a comment by
P.~Deligne.
\end{par}

\vspace{4mm}
\begin{par}{\bf Acknowledgments.}
I am indebted to the University of Regensburg, who kindly supported me for quite
some time now without ever seriously complaining. I am grateful to Tam\'as
Szamuely who suggested valuable improvements to earlier versions of the text,
and to Pierre Deligne who contributed an essential complement and generously
permitted me to propagate it. This work was initiated in July 2010 while
attending a workshop on Tannakian categories in Lausanne, I wish to thank the
EPFL and in particular Varvara Karpova for kindest hospitality. 
\end{par}

\vspace{14mm}
\section{Complements and recollections on 1--motives}\label{Sec:RealizOf1Mot}

\begin{par}
We introduce some constructions related to 1--motives relevant to our goals, and
also recall some standard definitions. 
\end{par}

\vspace{4mm}
\begin{para}\label{Par:1MotIntro}
We start with some recollections on 1--motives. By a 1--motive $M$ over a scheme $S$ we
understand a diagram of commutative group schemes over $S$ of the form
$$M = \left[\begin{diagram} \setlength{\dgARROWLENGTH}{3mm}
\node[3]{Y}\arrow{s,l}{u}\\
\node{\!\!0\!}\arrow{e}\node{\!T\!}\arrow{e}\node{\!G\!}\arrow{e}\node{\!A\!}
\arrow{e}\node{\!0\!\!} \end{diagram}\right]$$
where $G$ is an extension of an abelian scheme $A$ by a torus $T$, and $Y$
\'etale locally constant, locally isomorphic to a finitely generated free
$\IZ$--module. A \emph{morphism of 1--motives} is a morphism of diagrams. The
\emph{weight filtration} of $M$ is the three step filtration $0 \subseteq
W_{-2}M \subseteq W_{-1}M \subseteq W_{0}M = M$  given by 
$$W_{-2}M = \left[\begin{diagram} \setlength{\dgARROWLENGTH}{3mm}
\node[3]{0}\arrow{s}\\
\node{\!\!0\!}\arrow{e}\node{\!T\!}\arrow{e,=}\node{\!T\!}\arrow{e}\node{\!0\!}
\arrow{e}\node{\!0\!\!} \end{diagram}\right] \qqet W_{-1}M =
\left[\begin{diagram} \setlength{\dgARROWLENGTH}{3mm} \node[3]{0}\arrow{s}\\
\node{\!\!0\!}\arrow{e}\node{\!T\!}\arrow{e}\node{\!G\!}\arrow{e}\node{\!A\!}
\arrow{e}\node{\!0\!\!} \end{diagram}\right]$$
This filtration is functorial in $M$. While the category of 1--motives is not an
abelian category, it is an exact category. That a sequence $0\to M\to M'\to
M''\to 0$ is exact means that the simple complex associated with the double
complex
$$\begin{diagram}
\setlength{\dgARROWLENGTH}{3mm}
\node{Y}\arrow{s}\arrow{e}\node{Y'}\arrow{s}\arrow{e}\node{Y''}\arrow{s}\\
\node{G}\arrow{e}\node{G'}\arrow{e}\node{G''} 
\end{diagram}$$
is fppf--acyclic. The quotients $M/W_i(M)$ exist, and we have in particular
$$\gr^W_\ast(M) = \left[\begin{diagram} \setlength{\dgARROWLENGTH}{3mm}
\node[3]{Y}\arrow{s,l}{0}\\
\node{\!\!0\!}\arrow{e}\node{\!T\!}\arrow{e}\node{\!T\!\oplus
\!A\!}\arrow{e}\node{\!A\!}\arrow{e}\node{\!0\!\!} \end{diagram}\right]$$
We will often identify 1--motives $M$ with two term complexes
$[Y\xrightarrow{\:\:u\:\:} G]$ placed in degrees $0$ and $1$, and accordingly
morphisms of 1--motives with morphisms of complexes.
\end{para}

\vspace{4mm}
\begin{para}\label{Para:IntroHodgeRealis}
With every 1--motive $M$ over $\IC$ is associated an integral mixed Hodge
structure $\Tnot M$, called the \emph{Hodge realisation} of $M$. The
construction of $\Tnot M$ goes as follows: The kernel of the exponential map
$\exp:\Lie G(\IC) \to G(\IC)$ is canonically isomorphic to the singular homology
group $H_1(G(\IC),\IZ)$. Consider then the pull--back diagram
$$\begin{diagram}
\setlength{\dgARROWLENGTH}{4mm}
\node{0}\arrow{e}\node{H_1(G(\IC),\IZ)}\arrow{s,=}\arrow{e}\node{T_0
M}\arrow{s}\arrow{e}\node{Y}\arrow{s,r}{u}\arrow{e}\node{0}\\
\node{0}\arrow{e}\node{H_1(G(\IC),\IZ)}\arrow{e}\node{\Lie
G(\IC)}\arrow{e,t}{\exp}\node{G(\IC)}\arrow{e}\node{0}
\end{diagram}$$
The group $\Tnot M$ is finitely generated and free. It depends functorially on
$M$ hence carries a weight filtration induced by the weight filtration of $M$.
The Hodge filtration on $\Tnot M \otimes \IC$ has only one nontrivial step which
is determined by the Hodge filtration on $H_1(A(\IC),\IC)$. We write $\Vnot M :=
\Tnot M \otimes \IQ$ for the corresponding rational Hodge structure. Deligne has
shown (\cite{Deli74}, \S 10.1), that the functor $\Tnot$ from the category of
1--motives over $\IC$ to the category of integral mixed Hodge structures is
fully faithful. In other words, the natural maps
$$\Hom_\IC(M_1,M_2) \to \Hom_{\rm MHS}(\Tnot M_1, \Tnot M_2)$$
is an isomorphism for all 1--motives $M_1$, $M_2$ over $\IC$. The construction
of $\Tnot M$ behaves well in families: If $M$ is a 1--motive over a smooth
complex variety $S$, then the family $(\Tnot M_s)_{s\in S}$ is a variation of
mixed Hodge structures.
\end{para}

\vspace{4mm}
\begin{para}\label{Para:Def:fhM}
Let $V$ be a rational mixed Hodge structure. The \emph{Mumford--Tate group $H^V$
of $V$} is the fundamental group of the Tannakian category generated by $V$
inside the Tannakian category of rational mixed Hodge structures. We identify
this group with an algebraic subgroup of $\GL_V$ via its natural, faithful
action on $V$ and write $\fh^M \subseteq \End_\IQ(\Vnot M)$ for its Lie algebra. 
The \emph{weight filtration on $\fh^V$} is the filtration given by 
$$W_i(\fh^V) = \{f\in \fh^V \tq f(W_nV) \subseteq W_{n+i}V\}$$
The \emph{Mumford--Tate group of a 1--motive $M$ over $\IC$} is the
Mumford--Tate group of the Hodge--realisation $V_0M$ of $M$.
\end{para}

\vspace{4mm}
\begin{para}\label{Para:IntroElladicRealis}
Let $\ell$ be a prime number and let $M$ be a 1--motive over a field $k$ of
characteristic $\neq \ell$ with algebraic closure $\overline k$ and absolute
Galois group $\Gamma := \Gal(\overline k|k)$. The finite $\Gamma$--modules
$$M[\ell^i] := H^0(M\otimes^\IL \IZ/\ell^i\IZ)(\overline k) = \frac{\{(y,x)\in Y \times
G(\overline k) \tq u(y)=\ell^ix\}}{\{(\ell^iy,u(y))\tq y\in Y\}}$$
form a projective system for varying $i$, and we define
$$\Tell M := \lim_{i\geq 0} M[\ell^i]  \qqet \Vell M := \Tell M
\otimes\IQ_\ell$$
The object $\Tell M$ is a finitely generated free $\IZ_\ell$--module equipped
with a continuous action of $\Gamma$. The construction of $\Tell M$ is
functorial in $M$, hence a weight filtration on $\Tell M$ whose graded quotients
are the ordinary Tate modules of $T$ and $A$, and $Y\otimes\IZ_\ell$. If $k$ is
finitely generated over its prime field, then the natural maps
$$\Hom_k(M_1,M_2)\otimes\IZ_\ell \to \Hom_\Gamma(\Tell M_1,\Tell M_2)$$
are isomorphisms, and for all but finitely many prime numbers $\ell$, the maps
$$\Hom_k(M_1,M_2)\otimes \IZ/\ell\IZ \to \Hom_\Gamma(M_1[\ell], M_2[\ell])$$
are isomorphisms as well. These statement generalise the theorems of Tate,
Zahrin and Faltings about homomorphisms between abelian varieties over finitely
generated fields (and can be deduced by d\'evissage from these theorems and the
Mordell--Weil theorem, see \cite{Jann94}, Theorem 4.6). The construction of
$\Tell M$ behaves well for 1--motives $M$ over a base scheme $S$ over which
$\ell$ is invertible. In that case, $\Tell M$ is a smooth $\ell$--adic sheaf on
$S$.
\end{para}

\vspace{4mm}
\begin{para}\label{Para:Def:flM}
Let $M$ be a 1--motive over a field $k$ of characteristic $\neq \ell$, and let
$\rho_\ell:\Gal(\overline k|k) \to \GL(\Tell M)$ be the associated Galois
representation. The image of $\rho_\ell$ is a closed subgroup of $ \GL(\Tell
M)$, hence has the structure of an $\ell$--adic Lie group. We denote by $\fl^M
\subseteq \End_{\IQ_\ell}(\Vell M)$ its Lie algebra.
\end{para}

\vspace{4mm}
\begin{para}
Let $M$ be a 1--motive over a field of characteristic zero $k$. The deRham
realisation of $M$ is a finite dimensional vector space over $k$, which is
constructed as follows: Among the extensions of $M$ by vector groups there is a
universal one, given by
$$\begin{diagram}
\setlength{\dgARROWLENGTH}{4mm}
\node[3]{Y}\arrow{s}\arrow{e,=}\node{Y}\arrow{s}\\
\node{0}\arrow{e}\node{\cExt(M,\IG_a[-1])^\ast}\arrow{e}\node{G^\natural}\arrow{
e}\node{G}\arrow{e}\node{0}
\end{diagram}$$
We set $\VdR(M)=\Lie G^\natural$. This is a finite dimensional vector group over
$k$ which depends functorially on $M$, hence the weight filtration on $M$
defines a weight filtration on $\VdR(M)$. We define the Hodge filtration on
$\VdR(M)$ by $F^0\VdR(M) := \ker(\Lie G^\natural \to \Lie G)$. If $M$ is a
1--motive over a smooth variety $S$ over $k$, then the deRham--realisation
defines a finitely generated locally free $\cO_S$--module. This module comes
equipped with a canonical integrable connection
$$\nabla:\VdR(M) \to \VdR(M)\otimes_{\cO_S}\Omega^1_{S/k}$$
called the \emph{Gauss--Manin connection} (see \cite{AnBe11} \S4.2 for a construction). If $M$ is given by an abelian variety $A$ over $S$, then $\VdR(A)$ identifies with the dual of $H^1_\updR(A/S)$, and the Gauss--Manin connection is the classical one constructed by Katz and Oda by \emph{loc.cit.}, Lemma 4.5.
\end{para}

\vspace{4mm}
\begin{para}
There exist canonical isomorphisms comparing the Hodge realisation of a
1--motive with the $\ell$--adic and the deRham realisation. Given a 1--motive
$M$ over a finitely generated extension $k$ of $\IQ$, a complex embedding
$\sigma:k\to \IC$ and an extension of $\sigma$ to an algebraic closure
$\overline k$ of $k$, these are isomorphisms
$$\Tnot(\sigma^\ast M) \otimes_{\IZ} \IZ_\ell \xrightarrow{\:\:\cong\:\:}\Tell M
\qqet \Tnot(\sigma^\ast M) \otimes_{\IZ} \IC  \xrightarrow{\:\:\cong\:\:} \VdR M
\otimes_k\IC$$
of $\IZ_\ell$--modules and of complex vector--spaces respectively, where
$\sigma^\ast M$ is the pull--back of $M$ to $\spec\IC$ via $\sigma$. These
isomorphisms are natural in $M$, hence in particular respect the weight
filtration, and work also for families.
\end{para}

\vspace{4mm}
\begin{par}
We now present a special family of 1--motives, which shows that two 1--motives
can be smoothly deformed into into each other if they have the same graded
pieces for the weight filtration, that is, if they are built from the same
torus, abelian variety and lattice.
\end{par}

\vspace{4mm}
\begin{prop}\label{Pro:ModulispaceMain}
Let $T$ be a torus, $A$ be an abelian scheme and $Y$ be a lattice over a scheme $S$, and
set $M_0 := [Y\xrightarrow{\:0\:}(T\oplus A)]$. The fppf--presheaf on $S$ given
by
$$(i:U \rightarrow S) \quad \mapsto \quad \frac{\mbox{1--Motives $M$ over $U$
with $\gr^W_\ast(M)= i^\ast M_0$}}{\mbox{Isomorphisms $\alpha$ with
$\gr^W_\ast(\alpha) = \id_{i^\ast M_0}$}}$$
is representable by a $S$--scheme $X_S(T,A,Y)$ which is smooth over $S$. In
particular, this presheaf is a sheaf. More precisely, the $S$--schemes
$X_S(T,A,0)$ and $X_S(0,A,Y)$ are abelian schemes over $S$, and $X_S(T,0,Y)$ is
a torus over $S$, and there are isomorphisms of sheaves
$$X_S(T,A,0) \cong \cHom(T^\vee,A^\vee)  \qquad X_S(0,A,Y) \cong \cHom(Y,A)
\qquad X_S(T,0,Y) \cong \cHom(Y,T)$$
where $T^\vee$ is the character group of $T$ and $A^\vee$ the abelian scheme
dual to $A$. There is a canonical morphism
$$X_S(T,A,Y) \to X_S(T,A,0) \times_S X_S(0,A,Y)$$
which gives $X_S(T,A,Y)$ the structure of a $X_S(T,0,Y)$--torsor on $X_S(T,A,0)
\times_S X_S(0,A,Y)$.
\end{prop}

\begin{proof}
The fppf presheaf on $S$ associating with $i:U\rightarrow S$ the group
$\Ext^1_U(i^\ast A, i^\ast T)$ is representable by an abelian scheme $p':X'\to S$
over $S$, by the Barsotti--Weil formula. This means that we have natural
bijections
$$\Mor_S(U,X') \quad \xrightarrow{\:\:\cong\:\:} \quad \frac{\mbox{Semiabelian
schemes on $U$, extensions of $i^\ast A$ by $i^\ast T$}}{\mbox{Isomorphisms
inducing the identity on $i^\ast T$ and $i^\ast A$ }}$$
where by a \emph{semiabelian scheme} we understand a a group scheme which is
globally an extension of an abelian scheme by a torus\footnote{This is
nonstandard terminology. By a semiabelian scheme over $S$ one usually
understands a group scheme over $S$ each of whose fibres is an extension of an
abelian scheme by a torus. This is the right thing to consider in order to study
degenerations of abelian schemes.} over $S$. Denoting by $\mathcal G'$ be the
semiabelian scheme over $X'$ corresponding via this bijection to the identity
map $\id_{X'}$, the above bijection is given by 
$$(f':U\to X') \:\:\mapsto\:\: {f'}^\ast \mathcal G'$$
Set $\mathcal Y' := {p'}^\ast Y$. The fppf presheaf on $X'$ associating with
$j:U \to X'$ the group $\Hom_U(j^\ast \mathcal Y', j^\ast \mathcal G')$ is
representable by a semiabelian scheme $q:X\to X'$ over $X'$, so we have natural
bijections
$$\Mor_{X'}(U,X) \quad \xrightarrow{\:\:\cong\:\:} \quad \mbox{Homomorphisms
$j^\ast \mathcal Y' \to j^\ast \mathcal G'$ of fppf--sheaves on $U$}$$
Define $\mathcal Y := q^\ast \mathcal Y'$ and $\mathcal G := q^\ast \mathcal
G'$, and let $\mathcal M := [u:\mathcal Y \to \mathcal G]$ be the 1--motive over
$X$ where $u$ is the morphism corresponding via this bijection to the identity
morphism $\id_X$. The above bijection is then given by sending a $j:U\rightarrow
X'$ to the 1--motive $j^\ast\mathcal M$. We claim that the scheme $X$,
considered as a scheme over $S$ via the composite $p := p'\circ q$, has the
required properties. Because $X'$ is smooth and connected over $S$ and $X$ is
smooth and connected over $X'$, the scheme $X$ is smooth and connected over $S$.
We now show that for every $S$--scheme $i:U\to S$ the natural map
$$\Mor_S(U,X) \quad \to \quad \frac{\mbox{1--Motives $M$ over $U$ with
$\gr^W_\ast(M)= i^\ast M_0$}}{\mbox{Isomorphisms $\alpha$ with
$\gr^W_\ast(\alpha) = \id_{i^\ast M_0}$}}$$
sending an $S$--morphism $(f:U\to X)$ to the 1--motive $f^\ast\mathcal M$ over
$U$ is a bijection. Indeed, to give an $S$--morphism $f$ of an $S$--scheme $U$
to $X$ is the same as to give an $S$--morphism $f':U\to X'$ and an
$X'$--morphism $g: U \to X$, where $U$ is now viewed as an $X'$--scheme via
$j=f'$:
$$\begin{diagram}
\node{U}\arrow{sse,b}{i}\arrow{se,t}{j=f'}\arrow[2]{e,t}{g=f}\node[2]{X}\arrow{sw,t}{q}\arrow{ssw,b}{p}\\
\node[2]{X'}\arrow{s,r}{p'}\\
\node[2]{S}
\end{diagram}$$
So, to give an $S$--morphism $f:U\to X$ is the same as to give an extension $G =
{f'}^\ast\mathcal G'$ of $i^\ast A$ by $i^\ast T$ on $U$ modulo appropriate
isomorphisms and a homomorphism of $i^\ast Y = j^\ast \mathcal Y'$ to $G =
j^\ast \mathcal G'$. This datum is exactly what a 1--motive $M$ over $U$ with
$\gr^\ast_W(M)=i^\ast M_0$ consists of, again modulo appropriate isomorphisms.
\end{proof}

\vspace{4mm}
\begin{rem}
Consider the case where $Y = \IZ$ and $T = \IG_m$. Then $X_S(T,A,Y)$ is a
$\IG_m$--bundle over $A\times A^\vee$, that is, an invertible sheaf. This sheaf
is the Poincar\'e sheaf. This shows that $X_S(T,A,Y)$ is not a group scheme,
except in the degenerate cases.
\end{rem}

\vspace{4mm}
\begin{rem}\label{Rem:HomExtAndComparison}
Let $M = [Y\to G]$ be a 1--motive over a scheme $S$, and write $M_A := M/W_{-2}M = [Y \to
A]$, notations being as in \ref{Par:1MotIntro}. We have already used that the
fppf--sheaves $\cHom(Y,G)$ and $\cExt^1(M_A,T)$ are representable by semiabelian
schemes over $S$. If $\ell$ is invertible on $S$, there are canonical
isomorphisms of $\ell$--adic sheaves
$$\Tnot\cHom(Y,G) \cong \Hom_{\IZ}(Y,\Tnot G) \qqet \Tell\cExt^1(M_A,T) \cong
\Hom_{\IZ_\ell}(\Tell M_A ,\Tell T)$$
on $S$, and similar isomorphisms of variations of Hodge--structures if $S$ is
smooth over $\IC$. These isomorphisms are compatible with the comparison
isomorphisms, meaning that the squares
$$\begin{diagram}
\setlength{\dgARROWLENGTH}{4mm}
\node{\Hom_\IZ(Y,\Tnot
G)\otimes\IZ_\ell}\arrow{s,l}{\cong}\arrow{e,t}{\cong}\node{\Hom_{\IZ_\ell}
(Y\otimes \IZ_\ell,\Tell G)} \arrow{s,l}{\cong}\node{\!\!\!\Hom_\IZ(\Tnot
M_A,\!\Tnot T)\!\otimes\!\IZ_\ell\!\!}\arrow{s,l}{\cong}\arrow{e,t} {\cong}
\node{\!\!\Hom_{\IZ_\ell}(\Tell M_A,\Tell T)} \arrow{s,l}{\cong}\\
\node{\Tnot\cHom(Y,
G)\otimes\IZ_\ell}\arrow{n}\arrow{e,t}{\cong}\node{\Tell\cHom(Y,
G)}\arrow{n}\node{\Tnot\cExt(M_A, T)\otimes\IZ_\ell}\arrow{n}
\arrow{e,t}{\cong}\node{\Tell\cExt(M_A, T)}\arrow{n}
\end{diagram}$$
commute.
\end{rem}

\vspace{14mm}
\section{Cohomological realisation of families of
1--motives}\label{Sec:GaloisInMT}

\begin{par}
In the previous section we have associated a $\IQ$--Lie algebra $\fh^M \subseteq
\End_\IQ(\Vnot M)$ with a 1--motive $M$ over $\IC$ (\ref{Para:Def:fhM}), and a
$\IQ_\ell$--Lie algebra $\fl^M \subseteq \End_{\IQ_\ell}(\Vell M)$ with a
1--motive $M$ over a field of characteristic $\neq \ell$ (\ref{Para:Def:flM}).
In this section we will show that if $M$ is a 1--motive over a field which is
finitely generated over $\IQ$, then the Lie algebra $\fl^M$ is contained in
$\fh^M\otimes \IQ_\ell$ via the comparison isomorphism. By naturality of the
comparison isomorphism, this inclusion is compatible with the weight filtration.
\end{par}

\vspace{4mm}
\begin{thm}\label{Thm:ElladicInMumfordTate}
Let $k$ be a field of finite transcendence degree over $\IQ$ and let
$\sigma:k\to \IC$ be an embedding. Let $M$ be a 1--motive over $k$, and identify
$\fh^{(\sigma^\ast M)}\otimes \IQ_\ell$ with a subalgebra of
$\End_{\IQ_\ell}(\Vell M)$ via the comparison isomorphism $\Vnot(\sigma^\ast M)
\otimes\IQ_\ell \cong \Vell M$. Then the Lie algebra $\fl^M$ is contained in
$\fh^{(\sigma^\ast M)} \otimes \IQ_\ell$.
\end{thm}

\begin{par}
For abelian varieties in place of $M$ this was shown by Deligne in \cite{Deli82}
(see also \cite{ChSch11}), essentially by proving that every Hodge cycle on an
abelian variety is an absolute Hodge cycle. For 1--motives, the corresponding
statement about absolute Hodge cycles was proven by J.-L. Brylinski
(\cite{Brylinski}, Th\'eor\`eme 2.2.5):
\end{par}

\vspace{4mm}
\begin{thm}[Brylinski, Deligne]\label{Thm:BrylinskiDeligne}
Let $M$ be a 1--motive over $k$. Every Hodge cycle of $M$ relative to some
embedding $\sigma: k\to \IC$ is an absolute Hodge cycle.
\end{thm}

\begin{par}
After recalling the notion of absolute Hodge cycles we will give a proof of
theorem \ref{Thm:BrylinskiDeligne}, and then show how the statement about Lie
algebras follows from it. The proof of Brylinski's Theorem consists essentially
of a deformation argument, so we will be concerned with families of 1--motives
and their realisations. The idea is to show that if $M_1$ and $M_2$ are
1--motives such that $M_1$ can be smoothly deformed to $M_2$, then the statement
of Theorem \ref{Thm:BrylinskiDeligne} holds for $M_1$ if it holds for $M_2$. We
have already seen in Proposition \ref{Pro:ModulispaceMain} that every 1--motive
$M$ can be smoothly deformed to a split 1--motive. For split 1--motives the
statement of the theorem \ref{Thm:BrylinskiDeligne} is true by Deligne's Theorem
on absolute Hodge cycles on abelian varieties.
\end{par}

\begin{par}
I have decided to include a proof of theorem \ref{Thm:BrylinskiDeligne} to make
the text more self contained on one hand, and on the other hand because the
proof I present here seems a little more natural to me than Brylinski's. Indeed,
Brylinski's deformation process consists of using Hodge realisations in order to
produce an analytic family of 1--motives deforming a given 1--motive to a
1--motive which is split up to isogeny, and then to make this family algebraic
using GAGA (\cite{Brylinski}, Lemme 2.2.8.6). Having Proposition
\ref{Pro:ModulispaceMain} at hand, we can avoid all this.
\end{par}

\vspace{4mm}
\begin{para}
We fix for this section a field $k$ of of finite transcendence degree over $\IQ$
with algebraic closure $\overline k$. In order to handle realisations of a
motive $M$ over $k$ simultaneously, we introduce 
$$\IA_k := k \times \Big(\IQ\otimes\prod_{\!\!\!\!\!\ell \:
\mathrm{prime}\!\!\!\!\!}\IZ_\ell\Big) \qquad\qquad \VIA(M) := \VdR M \times
\Big(\IQ\otimes\prod_{\!\!\!\!\!\ell \: \mathrm{prime}\!\!\!\!\!}\Tell M\Big)$$
So $\IA_k$ is a commutative $\IQ$--algebra, and $\VIA(M)$ is a finitely
generated $\IA_k$--module. This works also in families: Let $S$ be an integral,
regular scheme of finite type over $k$ and let $M$ be a 1--motive over $S$. Then
we can consider the sheaf $\IA_S$ on $S_\upet$, and so $\VIA(M)$ is naturally a
sheaf of $\IA_S$--modules. We will work with tensor spaces of $\VIA(M)$. For a
finite family of pairs of nonnegative integers ${\bf n} = (n_i,n_i')_{i\in I}$,
set
$$\VIA^{\bf n}(M) = \bigoplus_{i\in I}\Big(\VIA(M)^{\otimes n_i} \otimes
(V_\IA(M)^\ast)^{\otimes n_i'}\Big)$$
where $(-)^\ast = \Hom(-,\IA_S)$. We refer to global sections of $\VIA^{\bf
n}(M)$ as \emph{tensors}. For every embedding $\sigma:k\to\IC$ there is a
canonical isomorphism of $\IA_{\sigma^\ast S}$--module sheaves
$$\alpha^\sigma: \Vnot^{\bf n}(\sigma^\ast M)\otimes_\IQ\IA_{\sigma^\ast S}
\xrightarrow{\:\:\cong\:\:} V_\IA^{\bf n}(\sigma^\ast M)$$
on the complex variety $\sigma^\ast S$, where $\Vnot^{\bf n}(\sigma^\ast M)$ is
the corresponding tensor space of variations of Hodge structures. These sheaves
are local systems for the complex topology.
\end{para}

\vspace{4mm}
\begin{defn}\label{Def:AbsHodgeCycle}
Let $M$ be a 1--motive over $k$. A tensor $t \in\Gamma(k,V_\IA^{\bf n}(M))$ for
some ${\bf n} = (n_i,n_i')_{i\in I}$ is called \emph{Hodge cycle relative to an
embedding} $\sigma:k\to\IC$ if the following holds:
\begin{enumerate}
 \item There exists an element $t_0 \in \Vnot^{\bf n}(\sigma^\ast M)$ such that
$t = \alpha^\sigma(t_0 \otimes 1)$.
 \item The deRham component $t_\mathrm{dR}$ of $t$ belongs to $F^0(\VdR^{\bf n}
M)\cap W_0(\VdR^{\bf n} M)$.
\end{enumerate}
An element $t \in \VIA^{\bf n}(M)$ is called \emph{absolute Hodge cycle} if it
is a Hodge cycle relative to all embeddings $\sigma:k\to \IC$.
\end{defn}

\vspace{4mm}
\begin{para}
In other words, the Hodge cycles relative to $\sigma:k\to\IC$ are the image of
elements of bidegree $(0,0)$ in $\Vnot^{\bf n}(M_\sigma)$ under the comparison
isomorphism $\alpha^\sigma$. They form therefore a finite dimensional
$\IQ$--linear subspace of $\VIA^{\bf n}(\sigma^\ast M)$.
\end{para}

\vspace{4mm}
\begin{prop}[Deligne]\label{Pro:AnsoluteHodgeFamily}
Let $S$ be a smooth connected scheme over $\overline k$ and let $s_0, s_1$ be closed
points of $S$. Let $M$ be a 1--motive over $S$, and let $t \in \Gamma(S,
\VIA^{\bf n}(M)$ be a tensor. Suppose that $t_\mathrm{dR}$ is annihilated by the
Gauss--Manin connection and that $(t_\mathrm{dR})_s$ is in $F^0(\VdR^{\bf n}
M_s)$ at every point $s\in S$. If $t_{s_0}$ is an absolute Hodge cycle, then so
is $t_{s_1}$.
\end{prop}

\begin{proof}
\begin{par}
Under the assumptions of the proposition, we have to show that $t_{s_1}$ meets
the two conditions in Definition \ref{Def:AbsHodgeCycle}. We start with
condition (1). Fix an embedding $\sigma:k\to\IC$. We claim that for all $s\in S$
the natural maps
$$\Gamma(\sigma^\ast S,\Vnot^{\bf n}(\sigma^\ast M)) \to \Vnot^{\bf n}(\sigma^\ast M_s) \qqet \Gamma(S,\VIA^{\bf n}(M))\to \VIA^{\bf n}(M_s)$$
are injective. Indeed, $\Vnot^{\bf n}(\sigma^\ast M)$ is a local system of
finite dimensional $\IQ$--vector spaces on $\sigma^\ast S$, and $\Tell^{\bf
n}(M)$ is a locally constant $\ell$--adic sheaf on $S$, so for any $s\in
S(\overline k)$ the global sections of these sheaves can be regarded as the
fixed points of the respective fibres at $s$ under the monodromy action of the
\'etale fundamental group based at $s$:
$$\Gamma(\sigma^\ast S,\Vnot^{\bf n}(\sigma^\ast M)) \cong \Vnot^{\bf
n}(\sigma^\ast M_s)^{\pi_1^\upet(s,\sigma^\ast S)} \qqet \Gamma(S,\Tell^{\bf
n}(M)) \cong \Tell^{\bf n}(M_s)^{\pi_1^\upet(s,S)}$$
The map $\Gamma(S,\VdR^{\bf n}(M))\to \VdR^{\bf n}(M_s)$ is injective because
$\VdR^{\bf n}(M)$ is a finitely generated locally free $\cO_S$--module, so our
claim follows. Following Deligne, we consider now this diagram:
$$\begin{diagram}
\setlength{\dgARROWLENGTH}{5mm}
\node{\Gamma(\sigma^\ast S,\Vnot^{\bf n}(\sigma^\ast
M))}\arrow{s}\arrow{e}\node{\Gamma(\sigma^\ast S,\Vnot^{\bf n}(\sigma^\ast
M_\sigma))\otimes \IA)}\arrow{s}\arrow{e,t}{\cong}\node{\Gamma(S,\VIA^{\bf
n}(M))}\arrow{s}\\
\node{\Vnot^{\bf n}(\sigma^\ast M_s)}\arrow{e}\node{\Vnot^{\bf n}(\sigma^\ast
M_s)\otimes \IA}\arrow{e,t}{\cong}\node{\VIA^{\bf n}(M_s)}
\end{diagram}$$
The left hand horizontal maps are given by $x\mapsto x\otimes 1$, and the right
hand horizontal maps are the comparison isomorphisms. We have seen that the
vertical maps are injective. Consider the above diagram for $s=s_0$. We are
given $t \in \Gamma(S,\VIA^{\bf n}(M))$ and are told that its image in
$\VIA^{\bf n}(M_{s_0})$ comes from an element in $\Vnot^{\bf n}(\sigma^\ast
M_{s_0})$. The intersection of the images of $\Gamma(S,\VIA^{\bf n}(M))$ and
$\Vnot^{\bf n}(\sigma^\ast M_s)$ in $\VIA^{\bf n}(M_s)$ is exactly the image of
$\Gamma(\sigma^\ast S, \Vnot^{\bf n}(\sigma^\ast M))$ in $\VIA^{\bf n}(M_s)$ by
standard linear algebra (choose a $\IQ$--basis of $\IA$ containing $1$). So $t$
comes from a global section $t^h$ of $\Vnot^{\bf n}(\sigma^\ast M)$. The element
$t_{s_1} \in \VIA^{\bf n}(M_{s_1})$ comes thus from an element in $\Vnot^{\bf
n}(\sigma^\ast M_{s_1})$, namely from the image of $t^h$ in $\VIA^{\bf
n}(\sigma^\ast M_{s_1})$, and that is what condition (1) asks for.
\end{par}
\begin{par}
We now come to the second condition. We have $(t_\updR)_{s_1} \in F^0(V^{\bf
n}_\updR M_{s_1})$ by assumption. The Gauss--Manin connection is functorial,
hence preserves the weight filtration (\cite{AnBe11}, \S4.2). Since $t_\updR$ is
horizontal and $(t_\updR)_{s_0} \in W_0(V^{\bf n}_\updR M_{s_0})$ we must also
have $W_0(V^{\bf n}_\updR M_{s_1})$ as needed.
\end{par}
\end{proof}

\vspace{4mm}
\begin{cor}\label{Cor:HodgeIsAbsoluteHodge}
Let $S$ be a smooth connected $\overline k$--scheme and $M$ be a 1--motive over
$S$. Let $V$ be a local subsystem of a tensor space $V_0^{\bf n}(M)$ such that
$V_s$ consists of $(0,0)$--cycles for all $s\in S$ and of absolute Hodge cycles
for at least one $s_0\in S$. Then $V_s$ consists of absolute Hodge cycles for
all $s$.
\end{cor}

\begin{proof}
The proof is litterally the same as the proof of 2.15 in \cite{Deli82}. The
argument is the following: If $V$ is constant, every element of
$V_{s_0}$ extends to a global section of $V$, and we are done by Proposition
\ref{Pro:AnsoluteHodgeFamily}. In general, observe that $\gr^W_0 (V^{\bf
n}_0(M))$ has a polarisation, so there is a rational, positive definite bilinear
form on $\gr^W_0 (V^{\bf n}_0(M))$ which is compatible with the action of
$\pi^1(S,s_0)$. Hence the image of $\pi^1(S,s_0)$ in $\GL(V_{s_0})$ is finite.
After passing to a finite cover of $S$, the local system $V$ becomes constant,
and we are done.
\end{proof}

\vspace{4mm}
\begin{lem}\label{Lem:BrylinskiDeligneSplit}
Let $M_0 = [Y \xrightarrow{\:0\:}(T \oplus A)]$ be a split 1--motive over $k$.
Every Hodge cycle of $M_0$ relative to some embedding $\sigma: k\to \IC$ is an
absolute Hodge cycle.
\end{lem}

\begin{proof}
Without loss of generality we may assume that $T$ is split and that $Y$ is
constant. We can also assume that $A$ is not trivial. But then, all tensor
spaces associated with realisations of $M_0$ can be obtained as direct factors of tensor spaces associated
with realisations of  $A$, and we know by the main result in \cite{Deli82} that every Hodge cycle
for $A$ is an absolute Hodge cycle.
\end{proof}

\vspace{4mm}
\begin{proof}[Proof of Theorem \ref{Thm:BrylinskiDeligne}]
Let $M_1$ be a 1--motive over $k$, let $\sigma:k\to \IC$ be a complex embedding
and let $t_1 \in \VIA(M_1)$ be a Hodge cycle relative $\sigma$. We have to show
that $t_1$ is an absolute Hodge cycle. Set $M_0 := \gr^W_\ast(M_1) = [Y
\xrightarrow{\:0\:}(T \oplus A)]$ and consider the smooth connected scheme $X :=
X_k(T,A,Y)$ and the universal 1--motive $M$ on $X$ from Proposition
\ref{Pro:ModulispaceMain}. The 1--motives $M_0$ and $M_1$ are isomorphic to the
fibres of $M$ in $k$--rational points $x_0, x_1 \in X(k)$, so we can deduce
Theorem \ref{Thm:BrylinskiDeligne} from Corollary \ref{Cor:HodgeIsAbsoluteHodge}
and Lemma \ref{Lem:BrylinskiDeligneSplit}.
\end{proof}

\vspace{4mm}
\begin{par}
It remains to deduce Theorem \ref{Thm:ElladicInMumfordTate} from Theorem
\ref{Thm:BrylinskiDeligne}. We start with the the following proposition,
analogous to Proposition 2.9.b of \cite{Deli82}.
\end{par}

\begin{prop}\label{Pro:ActionOnAbsHdgFinite}
Let $M$ be a 1--motive over $k$ and define $C_{AH}^{\bf n}$ to be the subspace
of absolute Hodge cycles in the tensor space $\VIA^{\bf n}(M)$. The Galois group
$\Gamma:= \Gal(\overline k|k)$ leaves $C_{AH}^{\bf n}$ invariant, and the action
of $\Gamma$ on $C_{AH}^{\bf n}$ factors over a finite quotient of $\Gamma$.
\end{prop}

\begin{proof}
That $\Gamma$ leaves $C_{AH}^{\bf n}$ invariant is immediate from the definition
of absolute Hodge cycles. Then, observe that for any prime number $\ell$ the map
$C_{AH}^{\bf n} \to \Vell^{\bf n}(M)$ is injective, and that the subgroup $N$ of
$\Gamma$ fixing $C_{AH}^{\bf n}$ is closed. The quotient $\Gamma/N$ is a
profinite group, and can be identified with a subgroup of the countable group
$\GL(C_{AH}^{\bf n})$. Hence $\Gamma/N$ must be finite.
\end{proof}

\vspace{4mm}
\begin{para}\label{Par:MTAlternativeDesc}
In what follows, we will use the following alternative description of the
Mumford--Tate group. Let $V$ be a rational mixed Hodge structure, and denote
also by $V$ its underlying rational vector space. The algebraic group $\GL_V$
acts naturally on the tensor spaces
$$V^{\bf n} = \bigoplus_{i\in I}\Big(V^{\otimes n_i} \otimes (V^\ast)^{\otimes
n_i'} \Big)$$
and the Mumford-Tate group of $V$, which is a subgroup of $\GL_V$, leaves all
elements of bidegree $(0,0)$ fixed. It follows from \cite{Deli82}, Proposition
3.1.c and the remark following it, that conversely the Mumford--Tate group of
$V$ is the largest subgroup of $\GL_V$ which fixes all elements of bidegree
$(0,0)$ in all tensor spaces of $V$. Because $\GL_V$ is noetherian, there exists
a tensor space $V^{\bf n}$ such that the Mumford--Tate group of $V$ is the
stabiliser in $\GL_V$ of the bidegree $(0,0)$ elements in $V^{\bf n}$.
\end{para}

\vspace{4mm}
\begin{proof}[Proof of Theorem \ref{Thm:ElladicInMumfordTate}]
Let us fix a tensor space $\Vnot^{\bf n} (\sigma^\ast M)$ such that the
Mumford--Tate group $H^{(\sigma^\ast M)}$ of $\Vnot(\sigma^\ast M)$ is equal to
the stabiliser in $\GL_{\Vnot(\sigma^\ast M)}$ of elements of bidegree $(0,0)$
in $\Vnot^{\bf n} (\sigma^\ast M)$. Such a tensor space exists, as explained in
\ref{Par:MTAlternativeDesc}. Denote by $C_{AH}^{\bf n}$ the finite dimensional
$\IQ$--linear subspace of absolute Hodge cycles in $\VIA^{\bf n}(M)$. By Theorem
\ref{Thm:BrylinskiDeligne} this subspace is equal to the image in $\VIA^{\bf
n}(M)$ of elements in $\Vnot^{\bf n}(\sigma^\ast M)$ of bidegree $(0,0)$ via the
comparison isomorphism
$$\Vnot^{\bf n}(\sigma^\ast M) \otimes_\IQ\IA
\xrightarrow{\:\:\cong\:\:}\VIA^{\bf n}(M)$$
By Proposition \ref{Pro:ActionOnAbsHdgFinite} there is an open subgroup
$\Gamma'$ of $\Gamma$ such that the action of $\Gamma'$ on $C_{AH}^{\bf n}$ is
trivial. In particular the image of $\Gamma'$ in the automorphisms of
$\Vell^{\bf n}(M)$ fixes the images of elements in $\Vnot^{\bf n}(\sigma^\ast
M)$ of bidegree $(0,0)$ under the comparison isomorphism $\Vnot^{\bf
n}(\sigma^\ast M) \otimes\IQ_\ell \cong \Vell^{\bf n}(M)$. The image of
$\Gamma'$ in the group of $\IQ_\ell$--linear automorphisms of $\Vell M$ is
therefore contained in the $\IQ_\ell$--points of $H^{(\sigma^\ast M)}$, and
because $\Gamma'$ is of finite index in $\Gamma$ this shows that $\fl^M$ is
contained in $\fh^{(\sigma^\ast M)}\otimes\IQ_\ell$ as we wanted to show.
\end{proof}


\vspace{14mm}
\section{Construction of the unipotent motivic fundamental
group}\label{Sec:UnipoMotFG}

\begin{par}
In this section we construct the Lie algebra of the unipotent motivic
fundamental group of a 1--motive. A construction of this object in terms of
biextensions and cubist symmetric torsors was proposed by P.~Deligne
(\cite{Bert03}). Our construction is more elementary, but has the disadvantage
that it is not a priori clear why it should produce the right thing. Deligne has
remediated this, I have reproduced his comments in the appendix.
\end{par}
\begin{par}
We restrict ourselves to 1--motives defined over a field. On one hand, our main
theorems are about 1--motives over fields and on the other hand, the
construction over a more general base scheme would involve delicate and
unrelated questions about semiabelian group schemes.
\end{par}

\vspace{4mm}
\begin{para}
\begin{par}
The idea of the motivic fundamental group of a 1--motive is the following: Let
$k$ be a field, and suppose for a moment that there exists a Tannakian category
$\cM_k$ of mixed motives over $k$ with rational coefficients. Let $M \in \cM_k$
be a 1-motive and write $\angl{M}^\otimes$ for the Tannakian subcategory of
$\cM_k$ generated by $M$. The motivic Galois group $\pi_\upmot(M)$ of $M$ is
defined to be the Tannakian fundamental group of $\angl{M}^\otimes$. The weight
filtration $W_\ast$ on $M$ defines a filtration on the group $\pi_\upmot(M)$ and
also on its Lie algebra, which we denote by the same letter $W_\ast$ and also
call weight filtration. The first filtration step $W_{-1}\pi_\upmot(M)$ is the
unipotent radical of $\pi_\upmot(M)$, because pure motives are semisimple
objects. We are interested in its Lie algebra 
$$W_{-1}(\Lie \pi_\upmot (M)) = \Lie W_{-1}(\pi_\upmot (M))$$
This is a Lie algebra object in the category of motives whose underlying mixed
motive has weights $-1$ and $-2$. From the point of view of 1--motives, it is a
semiabelian variety, say $P(M)$, which is moreover equipped with a Lie algebra
structure. We want to construct this semiabelian variety.
\end{par}
\begin{par}
Our plan of action for this section is the following: Given a 1--motive $M$ we
will construct geometrically a semiabelian variety $P(M)$ and declare it to be
$W_{-1}(\Lie \pi_\upmot (M))$. To justify our declaration, we establish in
sections \ref{Sec:CompMotivicHodge} and  \ref{Sec:CompMotivicGalois} canonical
isomorphisms 
$$ W_{-1}(\fh^M) \to \Vnot P(M) \qqet  W_{-1}(\fl^M) \to \Vell P(M)$$
of rational Hodge structures and of Galois representations respectively. The
semiabelian variety $P(M)$ comes equipped with a Lie bracket, and these
isomorphisms are both compatible with Lie brackets, that is, they are
isomorphisms of Lie algebra objects. This structure is important, but not for
our construction. For the sake of completeness we discuss it in the next
section, where we also check that our construction coincides with Deligne's up
to a canonical isomorphism.
\end{par}
\end{para}

\vspace{4mm}
\begin{para}\label{Par:ConstructionU(M)}
Let $M = [u:Y\to G]$ be a 1--motive over a field $k$. We start with constructing
a semiabelian variety $U(M)$ over $k$, which will contain $P(M)$. Write $M_A :=
M/W_{-2}M = [Y\to A]$. The two semiabelian varieties $\cHom(Y,G)$ and
$\cExt^1(M_A,T)$ are extensions of the abelian varieties $\cHom(Y,A)$ and
$\cExt^1(A,T)$ respectively by the torus $\cHom(Y,T)$. We define a semiabelian
variety $U(M)$ by requiring the short sequence of fppf--sheaves on $k$
$$0\to\cHom(Y,T) \xrightarrow{\:\:(+,-)\:\:}\cHom(Y,G)\times \cExt^1(M_A,T)
\xrightarrow{\qquad}  U(M) \to 0$$
to be exact. The first arrow is given on points by sending $t$ to the pair
$(\iota_1(t),-\iota_2(t))$, where $\iota_1$ is obtained by applying $\cHom(Y,-)$
to the morphism $T\to G$ and where $\iota_2$ is obtained by applying
$\cExt^1(-,T)$ to the map $M_A \to Y[1]$. Representability of $U(M)$ by a
semiabelian variety is not a problem. The map $u$ corresponds to a $k$--rational
point $u$ of $\cHom(Y,G)$, and viewing $M$ as an extension of $M_A$ by $T$ we
also get a global section $\eta$ on $\cExt^1(M_A,T)$. Denote by $\overline u$
the image of $(u, \eta)$ in $U(M)(k)$.
\end{para}

\vspace{4mm}
\begin{defn}\label{Def:P(M)}
Let $M$ be a 1--motive over a field $k$. We write  $P(M)$ for the smallest
semiabelian subvariety of $U(M)$ which contains $n\overline u$ for some nonzero
$n\in \IZ$, and name it \emph{Lie algebra of the unipotent motivic fundamental
group of $M$}.
\end{defn}

\vspace{4mm}
\begin{par}
Alternatively, we could declare $P(M)$ to be the connected component of the
unity of the Zariski closure of $\IZ\overline u$. It is clear that the
construction of $P(M)$ is compatible with base change. We continue by checking
that that the realisations of $U(M)$ are canonically isomorphic to the weight
$(-1)$ part of the linear endomorphisms of the corresponding realisation of $M$,
thus showing that $U(M)$ is the right habitat for $P(M)$.
\end{par}

\vspace{4mm}
\begin{prop}\label{Pro:Alpha0AlphaEll}
Let $M$ be a 1--motive over $\IC$, or over a field $k$ of characteristic $\neq
\ell$. Respectively, there are canonical isomorphisms 
$$\alpha_0: \Vnot U(M) \xrightarrow{\:\:\cong\:\:} W_{-1}\End_\IQ(\Vnot M) \qqet
\alpha_\ell: \Vell U(M) \xrightarrow{\:\:\cong\:\:} W_{-1}\End_{\IQ_\ell}(\Vell
M)$$
of Hodge structures and of Galois representations.
\end{prop}

\begin{proof}
The constructions and verifications are analogous for the Hodge and the
$\ell$--adic realisations, so we only treat the case of $\ell$--adic
realisation. We identify $\Vell T$ and $\Vell G$ with subspaces of $\Vell M$,
and write $\pi_Y$ and $\pi_{M_A}$ for the canonical projections onto $Y \otimes
\IQ_\ell$ and $\Vell M_A$. There are natural isomorphisms of Galois
representations
$$\Vell\cHom(Y,G) \cong \Hom_{\IQ_\ell}(Y\otimes \IQ_\ell,\Vell G) \qqet
\Vell\cExt^1(M_A,T) \cong \Hom_{\IQ_\ell}(\Vell M_A,\Vell T)$$
We can therefore represent elements of $\Vell U(M)$ by pairs $(f,g)$ where
$f:Y\otimes \IQ_\ell \to \Vell G$ and $g:\Vell M_A \to \Vell T$ are
$\IQ_\ell$--linear functions. We set
$$\alpha_\ell(f,g) = f \circ \pi_Y + g \circ \pi_{M_A} \qquad \in
W_{-1}\End_{\IQ_\ell}(\Vell M)$$
This yields a well defined map. Indeed, two pairs $(f,g)$ and $(f',g')$
represent the same element of $\Vell U(M)$ if and only if there exists a
$\IQ_\ell$--linear function $h: Y\otimes \IQ_\ell \to \Vell T$ such that $f-f' =
h$ and $g-g' = - h \circ \pi_Y$. So we have 
$$(f-f') \circ \pi_Y + (g-g') \circ \pi_{M_A} = h \circ \pi_Y - h \circ \pi_Y =
0$$
The map $\alpha_\ell:\Vell U \to W_{-1}\End_{\IQ_\ell}(\Vell M)$ thus defined is
linear, and also Galois equivariant. An inverse to $\alpha_\ell$ can be obtained
as follows. Choose a $\IQ_\ell$--linear section $s$ of $\pi_Y:\Vell M \to \Vell
Y$ and a $\IQ_\ell$--linear retraction $r$ of the inclusion $\Vell T \to \Vell
M$. For $\gamma \in W_{-1}\End_\IQ(\Vell M)$ we set
$$\alpha^{-1}_\ell(\gamma) = (f-h,g)$$
where $f$,$g$ and $h$ defined by $f = \gamma \circ s$, $g \circ \pi_{M_A} = r
\circ \gamma$ and $h = r \circ \gamma \circ s$. This makes sense because we
have 
$$\gamma(\Vell M)\subseteq \Vell G \qqet \gamma(\Vell G) \subseteq \Vell T \qqet
\gamma(\Vell T) = \{0\}$$
by definition of the weight filtration on $\End(\Vell M)$. To check that
$\alpha^{-1}_\ell$ is an inverse to $\alpha_\ell$ is straightforward.
\end{proof}

\vspace{4mm}
\begin{par}
We end this section with a technical definition, which we will use later in
sections \ref{Sec:CompMotivicHodge} and \ref{Sec:CompMotivicGalois}:
\end{par}

\begin{defn}\label{Def:AllIsogenyTypes}
Let $G$ and $\widetilde G$ be semiabelian varieties over an algebraically closed
field $\overline k$. We say that \emph{$\widetilde G$ contains all isogeny types
of $G$} if there exists an integer $n \geq 0$ and a morphism with finite kernel
$\gr^W_\ast(G) \to \gr^W_\ast(\widetilde G)^n$.
\end{defn}

\vspace{4mm}
\begin{lem}\label{Lem:GContainsAllIsogTypes}
Let $M=[Y\to G]$ be a 1--motive over an algebraically closed field $\overline
k$. The semiabelian variety $G$ contains all isogeny types of $P(M)$.
\end{lem}

\begin{proof}
We have $P(M)\subseteq U(M)$ by definition, so we can as well show that
$\gr^W_\ast(G)$ contains all isogeny types of $U(M)$. Write $G$ as an extension
of an abelian variety $A$ by a torus $T$. We choose an isogeny $A^\vee\to A$ and
isomorphisms $Y\simeq \IZ^r$ and $T \simeq \IG_m^s$. These choices induce a
morphism
$$\gr_{-1}^WU(M) = \cHom(Y,A) \oplus \cExt^1(A,T) \simeq A^r \oplus (A^\vee)^s
\to A^{r+s}$$
with finite kernel. We get also an isomorphism $\gr_{-2}^WU(M) = \cHom(Y,T)
\simeq \IG_m^{r+s}$, hence we can find a morphism from $\gr_\ast^W(U(M))$ to
$\gr^W_\ast(G)^{r+s}$ with finite kernel as needed.
\end{proof}

\vspace{14mm}
\section{Comments on Lie structures}\label{Sec:CommentsOnLie}

\begin{par}
In this section we  explain the Lie algebra structure on $P(M)$ and compare our
construction of $P(M)$ with the construction presented in \cite{Bert03}. We will
not use this comparison later, and the reader who is only interested in the
proof of our main theorems can skip it. We fix a field $k$ of characteristic
zero. 
\end{par}

\vspace{4mm}
\begin{para}\label{Par:WhatIsLieBracket}
What is a Lie bracket on a semiabelian variety $G$ over $k$? Naively, that
should be an alternating bilinear map $G\times G \to G$ satisfying the Jacobi
identity, but there are no such maps except for the zero map. There are two ways
out, one via the theory of biextensions, the other via homological algebra. We
choose to formulate our constructions in terms of homological algebra, so a Lie
bracket should be a morphism
$$\beta: G[-1] \otimes^\IL G[-1] \to G[-1]$$
in the derived category of fppf--sheaves\footnote{Recall that we have chosen to
place the complexes $[Y\to G]$ associated with 1--motives in degrees $0$ and
$1$. With this convention the 1--motive $[\IZ\to 0]$ is a neutral object for the
tensor product of complexes, as it should be.} on $\spec k$, which is graded
antisymmetric and satisfies the Jacobi--identity. Instead of $\beta$ we can also
give its adjoint $\ad_\beta: G[-1] \to \IR\cHom(G,G)$. The object
$\IR\cHom(G,G)$ is homologically concentrated in degrees $0$ and $1$, and given
in degree $1$ by the sheaf $\cExt^1(A,T)$ which is representable by an abelian
variety. Therefore, $\beta$ is uniquely determined by a morphism of abelian
varieties
$$\ad_\beta:A\to \cExt^1(A,T)$$
Any morphism $A\to \cExt^1(A,T)$ yields a morphism $G[-1]\otimes^\IL G[-1] \to
G[-1]$, so it remains to express antisymmetry and the Jacobi identity for
$\beta$ in terms of its adjoint. The Jacobi--identity comes for free, because
$\beta$ factors over a map $A[-1] \otimes^\IL A[-1] \to T[-1]$, so the derived
algebra is contained in the centre. This shows in particular that the Lie
algebra object $(G,\beta)$ is necessarily nilpotent -- from the point of view of
motives this was already clear for weight reasons. Denoting by $T^\vee$ the
group of characters of $T$, the abelian variety dual to $\cExt(A,T)$ is
canonically isomorphic to $T^\vee \otimes A$, and the morphism dual to
$\ad_\beta$ is then given by a morphism $T^\vee \otimes A \to A^\vee$, or
equivalently by a homomorphism of Galois modules
$$\lambda: T^\vee \to \Hom_{\overline k}(A,A^\vee)$$
From this point of view the antisymmetry condition is easy to express: The image
of $\lambda$ must be contained in the subgroup of selfdual homomorphisms $A\to
A^\vee$, that is, $\lambda(\chi) = \lambda(\chi)^\vee$ must hold for all
$\chi\in T^\vee$. We make this our definition:
\end{para}

\vspace{4mm}
\begin{defn}\label{Def:LieBracketOnSemiab}
Let $G$ be a semiabelian variety over $k$, extension of an abelian variety $A$
by a torus $T$. Write  $A^\vee$ for the abelian variety dual to $A$ and $T^\vee$
for the group of characters of $T$. A \emph{Lie algebra structure on $G$} is a
homomorphism of Galois modules
$$\lambda: T^\vee \to \Hom_{\overline k}(A,A^\vee)$$
such that $\lambda(\chi) = \lambda(\chi)^\vee$ holds for all $\chi\in T^\vee$.
\end{defn}

\vspace{4mm}
\begin{para}
This definition makes sense over any base scheme in place of $k$. To give a Lie
algebra structure on the semiabelian variety $G$ is the same as to give a Lie
algebra on the associated split semiabelian variety $\gr_\ast^W(G) = T \oplus
A$. Given a Lie algebra structure $\lambda: T^\vee \to \Hom_{\overline
k}(A,A^\vee)$ and a realisation functor, say $\Vell$, one gets a map
$$\Vell A \otimes \Vell A \to \Vell T$$
which equips the vector space $\Vell G$ with the structure of a nilpotent Lie
algebra. The derived Lie algebra $[\Vell G,\Vell G]$ is contained in $\Vell T$,
and in fact equal to $\Vell(T')$, where $T'\subseteq T$ is the subtorus with
character group $T^\vee/\ker(\lambda)$ modulo torsion. A polarisation $\psi:A\to
A^\vee$ defines a Lie algebra structure on $\IG_m\oplus A$, and the Lie bracket
one obtains from this is the classical Weil pairing.
\end{para}

\vspace{4mm}
\begin{para}
Let $M$ be a 1--motive over $k$. The semiabelian variety $U(M)$ carries a
canonical Lie algebra structure. Set $U_T := \cHom(Y,T)$ and $U_A := \cHom(Y,A)
\oplus \cHom(T^\vee,A^\vee)$, so $U(M)$ is an extension of $U_A$ by $U_T$. The
dual $U_A^\vee$ of $U_A$ is canonically isomorphic to $(Y\otimes A^\vee) \oplus
(T^\vee \otimes A)$, and the character group of $U_T$ is canonically isomorphic
to $Y\otimes T^\vee$. So we can define
$$\lambda: (Y\otimes T^\vee) \to \Hom_{\overline k}(U_A,U_A^\vee) \qquad\qquad
\lambda(y\otimes\chi)(f,g) = \big(y\otimes g(\chi),\chi\otimes f(y)\big)$$
We leave it to the reader to check that we have indeed
$\lambda(y\otimes\chi)=\lambda(y\otimes\chi)^\vee$ for all $y\in Y$ and $\chi\in
T^\vee$, and that the induced Lie bracket on realisations, say on $\Vnot U(M)$,
is given by
$$[(f,g),(f',g')] = (g\circ f',-f'\circ g)$$
for all $(f,g),(f',g') \in \Hom_\IQ(Y\otimes \IQ,\Vnot G)\oplus \Hom_\IQ(\Vnot
M_A,\Vnot T)$. This shows in particular that the canonical isomorphisms
$\alpha_0$ and $\alpha_\ell$ from Proposition \ref{Pro:Alpha0AlphaEll} are
isomorphisms of Lie algebras.
\end{para}

\vspace{4mm}
\begin{para}
The abelian variety $U_A$ contains the special rational point $\overline v$
coming from the 1--motive $M$, the image of $\overline u$ in $U_A$. We can
recover $U$ from  $\lambda$, $\overline v$ and its graded pieces $U_T$ and
$U_A$. Indeed, the dual of $U$ is given by the morphism
$$(Y\otimes T^\vee) \to U_A^\vee \qquad\qquad (y\otimes\chi) \mapsto
\lambda(y\otimes\chi)(\overline v) = \big(y\otimes v^\vee(\chi),\chi\otimes
v(y)\big)$$
The same formula must then hold for the subvariety $P(M)$ of $U(M)$, whose dual
is a quotient of the 1--motive $[(Y\otimes T^\vee) \to U_A^\vee]$. This is what
is meant in \cite{Bert03} by saying that the unipotent radical of the Lie
algebra of $\pi_\upmot(M)$ is the semiabelian variety defined by the adjoint
action of the semisimplification of the Lie algebra of $W_{-1}\pi_\upmot(M)$ on
itself. 
\end{para}

\vspace{4mm}
\begin{para}
It remains to explain why $P(M)$ is a Lie subobject of $U(M)$. This is indeed
nothing special to $P(M)$ and $U(M)$, so let us consider any abelian variety
$A$, torus $T$, a Lie algebra structure 
$$\lambda:T^\vee\to\Hom_{\overline k}(A,A^\vee)$$
and a rational point $a\in A(k)$. The Lie algebra structure $\lambda$ and the
point $a$ define an extension $G$ of $A$ by $T$, namely the dual of the
1--motive
$$[w: T^\vee \to A^\vee] \qquad\qquad w(\chi)=\lambda(\chi)(a)$$
In this situation, the following is true:
\end{para}

\vspace{4mm}
\begin{prop}\label{Pro:SmallestLieSubobject}
Let $A'\subseteq A$ be the connected component of the algebraic subgroup of $A$
generated by $a$, and let $G' \subseteq G$ be any semiabelian subvariety whose
projection to $A$ equals $A'$. Then $G'$ is a Lie subobject of $G$.
\end{prop}

\begin{par}
\noindent For example if $g$ is a preimage of $a$ in $G$, then $G'$ could be the
connected component of the algebraic subgroup of $G$ generated by $g$. This is
what we have in our concrete situation $P(M) \subseteq U(M)$.
\end{par}

\begin{proof}[Proof of \ref{Pro:SmallestLieSubobject}]
We suppose without loss of generality that $a$ is a rational point of $A'$, so
$A'$ is the Zariski closure of $\IZ a$. Denote by $[w':{T'}^\vee \to {A'}^\vee]$
the 1--motive dual to $G'$. The dual of the inclusion $G'\to G$ is then a
commutative square
$$\begin{diagram}
\setlength{\dgARROWLENGTH}{4mm}
\node{T^\vee}\arrow{s,l}{\kappa^\vee}\arrow{e,t}{w}\node{A^\vee}\arrow{s,r}{
\iota^\vee}\\
\node{{T'}^\vee} \arrow{e,t}{w'}\node{{A'}^\vee}
\end{diagram}$$
with surjective vertical maps dual to the inclusions $\kappa:T'\to T$ and
$\iota:A'\to A$. To say that $G'$ is a Lie subobject of $G$ is to say that the
arrow $\lambda'$ in the diagram
$$\begin{diagram}
\setlength{\dgARROWLENGTH}{5mm}
\node{T^\vee}\arrow{s,l}{\kappa^\vee}\arrow[2]{e,t}{\lambda}\node[2]{\Hom_{
\overline k}(A,A^\vee)\hspace{-10mm}}\arrow{s,r}{f \mapsto \iota^\vee\circ
f\circ\iota}\\
\node{{T'}^\vee} \arrow[2]{e,t,--}{\lambda'}\node[2]{\Hom_{\overline
k}({A'},{A'}^\vee)\hspace{-10mm}}
\end{diagram}$$
exists. Let $\chi\in T^\vee$ be a character such that $\kappa^\vee(\chi)=0$. We
must show that the endomorphism $\iota^\vee \circ \lambda(\chi) \circ \iota$ of
$A'$ is trivial. Indeed, we have
$$\iota^{\vee}(\lambda(\chi)(a)) = \iota^\vee(w(\chi)) = w'(\kappa(\chi)) =
w'(0)=0$$
so $a \in A'$ belongs to the kernel of this endomorphism, and because $a$
generates $A'$ as an algebraic group, we must have $\iota^\vee \circ
\lambda(\chi) \circ \iota=0$.
\end{proof}


\vspace{14mm}
\section{Comparison of the motivic fundamental group with the Mumford--Tate
group}\label{Sec:CompMotivicHodge}

\begin{par}
In this section, we show that for every 1--motive $M$ over $\IC$ the mixes Hodge substructures  $W_{-1}\fh^M$ and $\Vnot P(M)$ of $\End_\IQ(\Vnot M)$ are equal, via the isomorphism $\alpha_0$ from Proposition \ref{Pro:Alpha0AlphaEll}. We write $\Gamma$ for the absolute Hodge group over $\IQ$, that is, the Tannakian fundamental group of the category of mixed rational Hodge structures. So $\Gamma$ is a group scheme over $\IQ$ which acts on the underlying rational vector space of every mixed rational Hodge structure, in such a way that we have an equivalence of categories
$$\mathrm{MHS}_\IQ \xrightarrow{\:\:\cong\:\:} \{\mbox{Finite dimensional $\IQ$--linear representations of $\Gamma$}\}$$
which is compatible with duals and tensor products. We look at mixed Hodge structures, and in particular at Hodge realisations of 1--motives, as $\IQ$--vector spaces together with an action of $\Gamma$. The Mumford--Tate group of a Hodge structure $V$ is then just the image of $\Gamma$ in $\GL_V$. For a 1--motive $M$ we write $\Gamma_M$ for the maximal subgroup of $\Gamma$ acting trivially on $\Vnot M$, in particular if notations are as in \ref{Par:1MotIntro}, then $\Gamma_{\gr^W_\ast\!\! M} = \Gamma_{T\oplus A \oplus Y}$ is the largest subgroup of $\Gamma$ acting trivially on all pure subquotients of $\Vnot M$.
\end{par}

\vspace{4mm}
\begin{para}\label{Par:DefinitionTheta0}
Let $M = [u:Y\to G]$ be a 1--motive over $\IC$ and set $U := U(M)$. The action
of $\Gamma$ on $\Vnot M$ is given by a group homomorphism $\rho_0:\Gamma \to
\GL_{\Vnot M}$, whose image is by definition the Mumford--Tate group of $M$. The
subgroup  $\Gamma_{\gr^W_\ast\!\! M}$ acts on $\Vnot M$ by unipotent
automorphisms, and we have 
$$\log(\rho_0(\gamma)) = (\rho_0(\gamma)-1) -
\textstyle\frac12(\rho_0(\gamma)-1)^2\qquad  \in W_{-1}\End(\Vnot M)$$
We have constructed a canonical isomorphism $\alpha_0: \Vnot U \to
W_{-1}\End(\Vnot M)$, and by composing we get a map $\vartheta_0 :=
\alpha_0^{-1}\circ\log\circ\rho_0$. The image of
$\vartheta_0:\Gamma_{\gr^W_\ast\!\! M}\to \Vnot U$ is a Lie subalgebra of $\Vnot
U$, isomorphic via $\alpha_0$ to the Lie algebra $W_{-1}\fh^M$. Here is the
picture:
$$\begin{diagram}
\setlength{\dgARROWLENGTH}{4mm}
\node[2]{\Gamma_{\gr^W_\ast\!\!
M}}\arrow[2]{s,r}{\log\circ\rho_0}\arrow{ssw,t}{\vartheta_0}\arrow{e,t}{
\subseteq}\node{\Gamma}\arrow{s,r}{\rho_0}\\
\node[3]{\GL_{\Vnot M}}\\
\node{\hspace{-6mm}\Vnot P(M) \subseteq \Vnot
U(M)}\arrow{e,t}{\alpha_0}\node{W_{-1}\End(\Vnot M)\hspace{-4mm}}
\end{diagram}$$
More explicitly, the map $\vartheta_0$ is given as follows: Choose a section
$s:Y\otimes \IQ \to \Vnot M$ and a retraction $r:\Vnot M \to \Vnot T$. Then,
$\vartheta_0(\gamma)$ is represented by the pair 
\begin{equation}\label{Eqn:Theta0Explicit1}
(f-h-\textstyle\frac12e,g) \quad \in \Hom_\IQ(Y\otimes \IQ,\Vnot G)\times
\Hom_\IQ(\Vnot M_A, \Vnot T)\tag{$\ast$}
\end{equation}
where $f,g,h$ and $e$ are given by
\begin{equation}\label{Eqn:Theta0Explicit2}
f(y) = \gamma s(y)-s(y) \qquad g(a) = r(\gamma \widetilde a- \widetilde a)
\qquad h = r \circ f \qquad e(y) = \gamma^2 s(y)-2\gamma s(y) +
s(y)\tag{$\ast\ast$}
\end{equation}
for all $y\in Y$ and $a \in \Vnot M_A$. In the second equality, $\widetilde a$
is any element of $\Vnot M$ mapping to $a\in \Vnot M_A$ and $\Vnot T$ is
understood to be contained in $\Vnot G$. The class of $(f-h-\frac12e,g)$ in
$\Vnot U(M)$ is independent of the choice of $s$ and $r$. The main result of
this section is:
\end{para}

\vspace{4mm}
\begin{thm}\label{Thm:Comparison0}
The image of the map $\vartheta_0:\Gamma_{\gr^W_\ast\!\! M}\to \Vnot U(M)$ is
equal to $\Vnot P(M)$. In other words, the map $\alpha_0$ induces an
isomorphism 
$$\Vnot P(M) \xrightarrow{\:\:\cong\:\:} W_{-1}\fh^M$$
of rational Hodge structures.
\end{thm}

\vspace{4mm}
\begin{para}\label{Par:ConstructionKappa0}
We begin with an auxiliary construction. Let $G$ be a semiabelian variety over
$\IC$, and let us construct a $\IQ$--linear map 
$$\kappa_0: G(\IC) \otimes\IQ \to H^1(\Gamma,\Vnot G)$$
as follows: Given a complex point $x\in G(\IC)$ we consider the 1--motive $M_x
:= [\IZ\xrightarrow{\:1\mapsto x\:} G]$. The weight filtration on $M_x$ induces
a long exact sequence of rational vector spaces starting with
$$0 \to (\Vnot G)^\Gamma \to (\Vnot M_x)^\Gamma \to \IQ
\xrightarrow{\:\:\partial\:\:} H^1(\Gamma,\Vnot G) \to \cdots$$
and we set $\kappa_0(x \otimes 1) = \partial(1)$. Explicitly, elements of the
integral Hodge realisation $\Tnot M_x$ are pairs $(v,n)\in\Lie G(\IC)\times \IZ$
with $\exp(v)=nx$. We define $\kappa_0(x \otimes 1)$ to be the class of the
cocycle $\gamma \mapsto \gamma(v,1)-(v,1)$ where $v\in \Lie G(\IC)$ is any
element such that $\exp(v)=x$.
\end{para}

\vspace{4mm}
\begin{prop}\label{Pro:Kappa0Properties}
The map $\kappa_0$ constructed in \ref{Par:ConstructionKappa0} is injective and
natural in $G$.
\end{prop}

\begin{proof}
Let $x \in G(\IC)$ be a complex point such that the cocycle $c:\gamma \mapsto
\gamma(v,1)-(v,1)$ is a coboundary, where $v\in \Lie G(\IC)$ is such that
$\exp(v)=x$. We have to show that $x$ is a torsion point. Indeed, since $c$ is a
coboundary, there exists an element $w \in \Vnot G \subseteq \Lie G(\IC)$ such that
$c(\gamma) = \gamma w - w$ for all $\gamma \in \Gamma$. Let $n>0$ be an integer
such that $nw\in\Tnot G = \ker(\exp)$. Then $(nv-nw,n)$ is a $\Gamma$--invariant
element of $\Tnot M_x$, hence the linear map $\IZ \to \Tnot [n\IZ\to G]$ sending
$1$ to $(nv-nw,n)$ is a morphism of Hodge structures. From this we get a
morphism of 1--motives
$$\begin{diagram}
\setlength{\dgARROWLENGTH}{4mm}
\node{\IZ}\arrow{s,l}{\cdot n}\arrow[2]{e}\node[2]{0}\arrow{s}\\
\node{n\IZ}\arrow[2]{e,t}{n\mapsto nx}\node[2]{G}
\end{diagram}$$
because integral Hodge realisation is a fully faithful functor by \cite{Deli74},
10.1.3. This shows that $nx=0$. Naturality of $\kappa_0$ follows from naturality
of the Hodge realisation.
\end{proof}

\vspace{4mm}
\begin{prop}\label{Pro:DividingOnePoint0}
Let $M =[\IZ \xrightarrow{\:\:u\:\:} G]$ be a 1--motive over $\IC$. The class
$\kappa_0(u(1) \otimes 1)$ restricts to zero in $H^1(\Gamma_M,\Vnot G)$.
\end{prop}

\begin{proof}
Choose $v\in\Lie G(\IC)$ such that $\exp(v)=u(1)$, so that the pair
$(v,1)\in\Lie G(\IC)\times\IZ$ defines an element of $\Vnot M$. By definition of
$\Gamma_M$ we have $\gamma(v,1) = (v,1)$ and hence $\kappa_0(\gamma)=0$ for all
$\gamma \in \Gamma_M$. 
\end{proof}

\vspace{4mm}
\begin{prop}\label{Pro:GroupInclusions0}
Let $M=[u:Y\to G]$ be a 1--motive over $\IC$ and consider the 1--motives 
$$M_U := [\IZ \xrightarrow{\:1\mapsto \overline u\:} U(M)] \qqet M_P := [n\IZ
\xrightarrow{\:n\mapsto n\overline u\:} P(M)]$$
where $n\geq 1$ is an integer such that the point $n\overline u$ of $U(M)$
belongs to $P(M)$ and $\overline u$ is as in Definition \ref{Def:P(M)}. The
inclusions $\Gamma_M \:=\: \Gamma_{M_U} \:\subseteq\: \Gamma_{M_P}$ hold in
$\Gamma$.
\end{prop}

\begin{proof}
We write $\angl{V}$ for the Tannakian subcategory of the category of rational
Hodge structures generated by a Hodge structure $V$. We have to show
$$\mathrm{a)}\quad \Vnot M \in \angl{\Vnot M_U} \qquad \mathrm{b)}\quad \Vnot
M_U \in \angl{\Vnot M} \qquad \mathrm{c)}\quad \Vnot M_P \in \angl{\Vnot M_U}$$
For a), consider the morphisms of 1--motives given by the following commutative
squares:
$$\begin{diagram}
\setlength{\dgARROWLENGTH}{4mm}
\node{\IZ}\arrow{s,=}\arrow[2]{e,t}{1\mapsto \overline
u}\node[2]{U(M)}\arrow{s,r}{\mathrm{proj}} \node[3]{Y\otimes\IZ}
\arrow{s,l}{\cong} \arrow[3]{e,t}{y\otimes 1\mapsto y\otimes u}\node[3]{Y\otimes
\cHom(Y,G)\hspace{-20mm}}\arrow{s,r}{y\otimes f \mapsto f(y)}\\
\node{\IZ}\arrow[2]{e,b}{1\mapsto
u}\node[2]{\cHom(Y,G)}\node[3]{Y}\arrow[3]{e,b}{u}\node[3]{G}
\end{diagram}$$
Both morphisms induce surjective morphisms of Hodge structures. The left hand
diagram shows that $\Vnot[\IZ\to \cHom(Y,G)]$ belongs to $\angl{\Vnot M_U}$,
hence also the Hodge structure
$$(Y\otimes\IQ) \otimes \Vnot[\IZ\to \cHom(Y,G)] \cong \Vnot[Y\otimes\IZ \to
Y\otimes\cHom(Y,G)]$$
The right hand morphism shows that also $\Vnot M$ belongs to $\angl{\Vnot M_U}$.
The verification of b) is similar, here we consider the morphism of 1--motives
given by
$$\begin{diagram}
\setlength{\dgARROWLENGTH}{4mm}
\node{\IZ}\arrow{s,l}{1\mapsto(\id,\id)}\arrow[7]{e,t}{1\mapsto
(u,\eta)}\node[7]{\cHom(Y,G)\oplus\cExt^1(M_A,T)}\arrow{s,=}\\
\node{\cHom(Y,Y)\oplus\cHom(T,T)}\arrow[7]{e,b}{(f,g)\mapsto (u\circ f, g_\ast
\eta)}\node[7]{\cHom(Y,G)\oplus\cExt^1(M_A,T)}
\end{diagram}$$
where $\eta\in \Ext^1_k(M_A,T)$ is the extension class defined by $M$. This
diagram induces an injection of Hodge structures. The Hodge structure associated
with the lower row is isomorphic to the direct sum of $\Hom_\IQ(Y\otimes\IQ,
\Vnot M)$ and $\Hom_\IQ(\Vnot M,\Vnot T)$, hence belongs to $\angl{\Vnot M}$.
Hence also the Hodge structure associated with the upper row belongs to
$\angl{\Vnot M}$, and $\Vnot M_U$ is a quotient of this Hodge structure by
definition of $U(M)$. Finally, c) is obvious since $\Vnot M_P$ is a substructure
of $\Vnot M_U$.
\end{proof}

\vspace{4mm}
\begin{cor}\label{Cor:HodgeDividing}
Let $M=[u:Y\to G]$ be a 1--motive over $\IC$ and set $\widetilde G := \gr^W_\ast\!\! G =  T \oplus A$. Let $n\geq 1$ be an integer such that the point $n\overline u$
of $U(M)$ belongs to $P(M)$. For every $\psi\in\Hom_\IC(P(M),\widetilde G)$, the
cohomology class $\kappa_0(\psi(n\overline u)\otimes 1)\in H^1(\Gamma,
\Vnot\widetilde G)$ restricts to zero in $H^1(\Gamma_M, \Vnot\widetilde G)$.
\end{cor}

\begin{proof}
Set $P := P(M)$, define 1--motives $M_P := [\IZ  \xrightarrow{\:1\mapsto
n\overline u\:} P]$ and $M_\psi := [\IZ  \xrightarrow{\:1\mapsto \psi(n\overline
u)\:} \widetilde G]$. The pair $(\id_\IZ,\psi)$ defines a morphism of 1--motives
$M_P\to M_\psi$, and we find a commutative diagram
$$\begin{diagram}
\setlength{\dgARROWLENGTH}{4mm}
\node{P(\IC) \otimes\IQ}\arrow{s,l}{\psi\otimes
1}\arrow{e,t}{\kappa_0}\node{H^1(\Gamma,\Vnot P)}\arrow{s}\arrow{e,t}{\res}
\node{H^1(\Gamma_M, \Vnot P)}\arrow{s}\\
\node{\widetilde G(\IC) \otimes\IQ}\arrow{e,t}{\kappa_0}\node{H^1(\Gamma,\Vnot
\widetilde G)}\arrow{e,t}{\res}\node{H^1(\Gamma_M,\Vnot \widetilde G)}
\end{diagram}$$
where all vertical maps are induced by $\psi$. The restriction map
$H^1(\Gamma,\Vnot P) \to H^1(\Gamma_M, \Vnot P)$ factors over $H^1(\Gamma_{M_P},
\Vnot P)$ by Proposition \ref{Pro:GroupInclusions0}, hence $\kappa_0(n\overline
u\otimes 1)$ maps to zero in $H^1(\Gamma_M, \Vnot P)$ by Proposition
\ref{Pro:DividingOnePoint0}. 
\end{proof}

\vspace{4mm}
\begin{para}\label{Par:SetupVartheta0OnPurePart}
With the help of the map $\kappa_0$ and its properties we established so far, we
can show that the image of $\vartheta_0$ is contained in $\Vnot P(M)$. Let $M =
[u:Y\to G]$ be a 1--motive over $\IC$ and write $M_A := M/W_{-2}M = [Y\to A]$
and $U := U(M)$. Let 
$$\pi: U \to U_A := \cHom(Y,A) \oplus \cExt^1(A,T)$$ 
be the projection onto the abelian quotient $U_A$ of $U$ and let $\iota$ be the
inclusion of $\Gamma_{G \oplus M_A}$ into $\Gamma_{\gr^W_\ast\!\! M}$. We
consider the two composition maps
$$U(\IC)\otimes\IQ \xrightarrow{\:\:\kappa_0\:\:} H^1(\Gamma_{\gr^W_\ast\!\!
M},\Vnot U) \xrightarrow{\:\:(\Vnot\pi)_\ast\:\:} H^1(\Gamma_{\gr^W_\ast\!\!
M},\Vnot U_A) \cong \Hom(\Gamma_{\gr^W_\ast\!\! M},\Vnot U_A)$$
and
$$U(\IC)\otimes\IQ \xrightarrow{\:\:\kappa_0\:\:} H^1(\Gamma_{\gr^W_\ast\!\!
M},\Vnot U) \xrightarrow{\:\:\iota^\ast\:\:} H^1(\Gamma_{G \oplus M_A},\Vnot U)
\cong \Hom(\Gamma_{G \oplus M_A},\Vnot U)$$
These send $\overline u \otimes 1$ to the homomorphisms $(\Vnot\pi)\circ
\kappa_0(\overline u \otimes 1)$ and $\kappa_0(\overline u \otimes 1)\circ
\iota$ respectively. Here we have used that $\Gamma_{\gr^W_\ast M}$ acts
trivially on $\Vnot U_A$ and that $\Gamma_{G \oplus M_A}$ acts trivially on
$\Vnot U$. The following lemma explains the relation between $\vartheta_0$ and
$\kappa_0$:
\end{para}

\vspace{4mm}
\begin{lem}\label{Lem:Vartheta0OnPurePart} 
Notations being as in \ref{Par:SetupVartheta0OnPurePart}, the equalities
$$(\Vnot\pi)\circ \vartheta_0 = (\Vnot\pi)\circ \kappa_0(\overline u \otimes 1)
\qqet \vartheta_0 \circ \iota = \kappa_0(\overline u \otimes 1)\circ \iota$$
hold in $\Hom(\Gamma_{\gr^W_\ast\!\! M},\Vnot U_A)$ and in $\Hom(\Gamma_{G
\oplus M_A},\Vnot U)$ respectively.
\end{lem}

\begin{proof}
\begin{par}
Write $M_U = [\IZ\xrightarrow{\:1\mapsto \overline u\:} U]$ and choose $x\in\Lie
U$ with $\exp(x)=\overline u$. We can represent $x$ by a pair
$$(x_f,x_g) \in \Lie\cHom(Y, G) \oplus \Lie\cExt^1(M_A, T)$$
with $\exp x_f = u$ and $\exp x_g = \eta$, the class of $M$ in $\Ext^1(M_A,T)$.
Regarding $x_f$ as a group homomorphism $x_f:Y\to \Lie G$, we get 
a section $s:Y\otimes \IQ \to \Vnot M$ defined by $s(y) = (x_f(y),y)$. Regarding
$x_g$ as a group homomorphism $x_g:T^\vee \to \Lie(M_A^\vee)$ we get a section
$T^\vee \otimes \IQ \to \Vnot M^\vee$ defined by $\chi \mapsto
(x_g(\chi),\chi)$. The linear dual of this section is a retraction $r: \Vnot M
\to \Vnot T$. With the help of the section $s$ and the retraction $r$ we can
describe $\vartheta_0$ explicitly as in \ref{Par:DefinitionTheta0}, so
$\vartheta_0(\gamma)$ is the class of $(f-h-\frac12e,g)$ where $f,g,h$ and $e$
depend on $s,r$ and $\gamma$ according to equation $(\ast\ast)$ in
\ref{Par:DefinitionTheta0}.
\end{par}
\begin{par}
Having this explicit description at hand, we can verify the left hand formula of
the lemma. Applying its left hand term to $\gamma \in \Gamma_{\gr^W_\ast\!\! M}$
we get the element of $\Vnot U_A$ given by the pair
$$\Vnot\pi\textstyle\big(f-h-\frac12e,g\big) = \big(\Vnot\pi \circ f
,\Vnot\pi\circ g\big) \in \Hom(Y\otimes \IQ,\Vnot A) \oplus \Hom(\Vnot A,\Vnot
T)$$
On the right hand side we find, using the definition of $\kappa_0$
$$(\Vnot\pi)\circ \kappa_0(\overline u \otimes 1)(\gamma) =
(\Vnot\pi)(\gamma(x,1)-(x,1)) = \gamma(\pi x,1)-(\pi x,1) \in \Vnot U_A$$
where $\pi x$ is the image of $x$ in $\Lie U_A$. This element is given by the
pair
$$\big(\gamma(\pi x_f,1)-(\pi x_f,1) , \gamma(\pi x_g,1)-(\pi x_g,1)\big) \in
\Vnot\cHom(Y,A)\oplus\Vnot\cExt^1(M_A,T)$$
which if viewed as an element of $\Hom(Y\otimes \IQ,\Vnot A) \times \Hom(\Vnot
A,\Vnot T)$ equals $(\Vnot\pi \circ f ,\Vnot\pi\circ g)$.
\end{par}
\begin{par}
We now come to the second formula. Let us fix an element $\gamma \in \Gamma_{G
\oplus M_A}$, so $\gamma$ acts trivially on $\Vnot G$ and on $\Vnot M_A$. Hence
we have $f=g=h$ and $e=0$, so $\vartheta_0(\gamma)$ is given by the homomorphism
$$h\in\Hom(Y\otimes\IQ, \Vnot T) \qquad h(y)=\gamma s_f(y)-s_f(y)$$
Under the canonical isomorphism $\Vnot\cHom(Y,G) \cong \Hom(Y\otimes\IQ,\Vnot
G)$ this homomorphism $h$ corresponds to the element
$$\gamma(x_f,1)-(x_f,1) = \gamma(x,1)-(x,1)$$
which equals $\kappa_0(\overline u\otimes 1)(\gamma)$ by definition of
$\kappa_0$.
\end{par}
\end{proof}

\vspace{4mm}
\begin{lem}\label{Lem:SSDetectionSub0} 
Let $0\to T \to G  \xrightarrow{\:\:\pi\:\:} A \to 0$ be a semiabelian variety
over $\IC$, let $G' \subseteq G$ be a semiabelian subvariety and let $V
\subseteq \Vnot G$ be a Hodge substructure. If the inclusions $\pi(V) \subseteq
\Vnot\pi(G')$ and $V\cap \Vnot T \subseteq \Vnot G'$ hold, then $V$ is contained
in $\Vnot G'$. (Think of it as exactness of $\gr^W_\ast$).
\end{lem}

\begin{proof}
Consider the following diagram with exact rows:
$$\begin{diagram}
\setlength{\dgARROWLENGTH}{4mm}
\node{0}\arrow{e}\node{V\cap\Vnot
T}\arrow{s,l}{0}\arrow{e}\node{V}\arrow{s}\arrow{e}\node{\pi(V)}\arrow{s,l}{0}
\arrow{e}\node{0}\\
\node{0}\arrow{e}\node{\Vnot(T/(T\cap
G'))}\arrow{e}\node{\Vnot(G/G')}\arrow{e}\node{\Vnot(A/\pi(G'))}\arrow{e}\node{0
}
\end{diagram}$$
The left and right vertical maps are zero by hypothesis, and we have to show
that the middle vertical map is zero as well. This follows by diagram chase,
using that there are no nonzero morphisms of Hodge structures $\pi(V)\to
\Vnot(T/(T\cap G'))$. Indeed, $\pi(V)$ is pure of weight $-1$ and
$\Vnot(T/(T\cap G'))$ is pure of weight $-2$.
\end{proof}

\vspace{4mm}
\begin{prop}\label{Pro:Theta0InVnotP} 
The image of the map  $\vartheta_0:\Gamma_{\gr^W_\ast\!\! M}\to \Vnot U$ is
contained in $\Vnot P$.
\end{prop}

\begin{proof}
This is a consequence of Lemmas \ref{Lem:Vartheta0OnPurePart} and
\ref{Lem:SSDetectionSub0}. Indeed, by Lemma \ref{Lem:SSDetectionSub0} it is
enough to check that the inclusions
$$(\Vnot\pi)(\im(\vartheta_0)) \subseteq \Vnot \pi(P) \qqet
\im(\vartheta_0)\cap\Vnot U_T \subseteq \Vnot P$$
hold. An element $\gamma\in \Gamma_{\gr^W_\ast\!\! M}$ belongs to $\Gamma_{G
\oplus M_A}$ if and only if it acts trivially on $\Vnot U_A$. Hence by Lemma
\ref{Lem:Vartheta0OnPurePart}, these inclusions are the same as
$$\im((\Vnot\pi) \circ \kappa_0(\overline u\otimes 1)) \subseteq \Vnot \pi(P)
\qqet \im(\kappa_0(\overline u\otimes 1) \circ\iota) \subseteq \Vnot P$$
But since $n\overline u \in P$ for some nonzero $n$, the class
$\kappa_0(\overline u\otimes 1)$ comes from a cocycle which takes values in
$\Vnot P$.
\end{proof}

\vspace{4mm}
\begin{lem}\label{Lem:HSSSIso0} 
Let $M=[Y\to G]$ be a 1--motive over $\IC$ and set $\widetilde G := \gr^\ast_WG
= T\oplus A$. The map
$$H^1(\fh^M,\Vnot \widetilde G) \to \Hom_{\fh^M}(W_{-1}\fh^M,\Vnot \widetilde
G)$$
given by restriction of cocycles is an isomorphism.
\end{lem}

\begin{proof}
The Hochschild--Serre spectral sequence for Lie algebra cohomology associated
with the Lie algebra extension $0\to W_{-1}\fh^M \to \fh^M \to \fh^{\widetilde
G} \to 0$ yields the following exact sequence in low degrees
$$H^1(\fh^{\widetilde G},\Vnot \widetilde G) \to H^1(\fh^M,\Vnot \widetilde G)
\to \Hom_{\fh^M}(W_{-1}\fh^M,\Vnot \widetilde G) \to H^2(\fh^{\widetilde
G},\Vnot \widetilde G)$$
so it suffices to show that the first and last term in this sequence vanish. To
do so, it suffices by Sah's lemma to show that there exists a central element
$x\in \fh^{\widetilde G}$ which acts as an automorphism on $\Vnot \widetilde G$.
But this is clear since $\fh^{\widetilde G}$ contains an element which acts as
the identity on $\Vnot A$ and as multiplication by $2$ on $\Vnot T$.
\end{proof}

\vspace{4mm}
\begin{lem}\label{Lem:InjectViaDividing0} 
Let $M=[u:Y\to G]$ be a 1--motive over $\IC$ and set $\widetilde G := \gr^W_\ast G =  T \oplus A$. The map $\alpha_0^\ast: \Hom_\Gamma(\Vnot P(M), \Vnot\widetilde G) \to \Hom_\Gamma(W_{-1}\fh^M, \Vnot\widetilde G)$ given by $\alpha_0^\ast(f)=f\circ\alpha_0^{-1}$ is injective.
\end{lem}

\begin{proof}
Set $P := P(M)$ for brevity. We will construct in a first step another injective map $\beta_0:\Hom_\Gamma(\Vnot P, \Vnot\widetilde G)\to \Hom_\Gamma(W_{-1}\fh^M, \Vnot \widetilde G)$, and prove in a second step that the equality $\alpha_0^\ast = \beta_0$ holds. For the construction of $\beta_0$ we use the following diagram
$$\begin{diagram}
\setlength{\dgARROWLENGTH}{4mm}
\node[3]{\Hom_\IC(P,\widetilde G)\otimes\IQ}\arrow{s,r}{(\ast)}\arrow{sw,--}\\
\node{0}\arrow{e}\node{H^1(\fh^M,\Vnot \widetilde
G)}\arrow{s,l}{\cong}\arrow{e}\node{H^1(\Gamma,\Vnot\widetilde
G)}\arrow{e,t}{\res}\node{H^1(\Gamma_M, \Vnot\widetilde G)}\\
\node[2]{\Hom_\Gamma(W_{-1}\fh^M,\Vnot\widetilde G)}
\end{diagram}$$
where the map $(\ast)$ is $\IQ$--linear and sends $\psi \otimes 1$ to $\kappa_0(\psi(n\overline u) \otimes n^{-1})$ for some integer $n$ such that $n\overline u \in P(k)$. By Corollary \ref{Cor:HodgeDividing} the map $(\ast)$ lifts to $H^1(\fh^M,\Vnot \widetilde G)$ as indicated. The isomorphism on the left is given by Lemma \ref{Lem:HSSSIso0}. Let $\beta_0$ be the composition
$$\beta_0 : \Hom_\Gamma(\Vnot P, \Vnot\widetilde G) \cong \Hom_\IC(P,\widetilde G)\otimes\IQ \to \Hom_\Gamma(W_{-1}\fh^M,\Vnot\widetilde G)$$
The map $\beta_0$ is injective because $(\ast)$ is so. Indeed, let $\psi\in\Hom(P,\widetilde G)$ be a homomorphism such that $\kappa_0(\psi(n\overline u) \otimes n^{-1}) = 0$. Then, since $\kappa_0$ is injective by Proposition \ref{Pro:Kappa0Properties} we have $\psi(n\overline u)=0$ and hence $\ker\psi$ is a subgroup of $P$ containing $n\overline u$. We must then have $\ker\psi = P$ by definition of $P$, so $\psi=0$. It remains to check that we have $\alpha_0^\ast = \beta_0$. Because all maps are $\IQ$--linear, we only have to check that for every $\psi\in\Hom(P,\widetilde G)$ and every $\gamma\in \Gamma_{\gr^W_\ast\!\! M}$ the equality
$$\alpha_0^\ast(V_0\psi)(\log\rho_0(\gamma)) = \beta_0(V_0\psi)(\log\rho_0(\gamma))$$
holds in $\Vnot\widetilde G$. The left hand side is equal to $\Vnot\psi(\vartheta_0(\gamma))$. Let $w\in\Lie U(\IC)$ be an element with $\exp(w)= \overline u$, and set $v := \psi(nw) \in \Lie\widetilde G$. We then have $\exp(v)=\psi(n\overline u)$ and using Lemma \ref{Lem:Vartheta0OnPurePart}
$$\textstyle \beta_0(\Vnot\psi)(\log\rho_0(\gamma)) = \kappa_0(\psi(n\overline u)\! \otimes\! \frac 1n)(\gamma) = \frac 1n(\gamma(v,\!1)-(v,\!1)) = \Vnot\psi(\gamma ( w,\!1)-( w,\!1)) = \Vnot\psi(\vartheta_0(\gamma))$$
as we wanted to show.
\end{proof}

\vspace{4mm}
\begin{lem}\label{Lem:DetectionOnSplit0}
Let $G$ be a semiabelian variety over $\IC$ and let $\widetilde G$ be a split
semiabelian variety containing all isogeny types of $G$ (Definition
\ref{Def:AllIsogenyTypes}). Let $V$ be a Hodge substructure of $\Vnot G$. If the
restriction map
$$\Hom_\Gamma(\Vnot G, \Vnot \widetilde G) \to \Hom_\Gamma(V, \Vnot \widetilde
G)$$
is injective, then $V$ is equal to $\Vnot G$.
\end{lem}

\begin{proof}
In the case where $G$ is an abelian variety or a torus, this is clear by
semisimplicity of the category of pure polarisable rational Hodge structures.
The general case can be proved by d\'evissage, writing $G$ as an extension of an
abelian variety by a torus. We give a detailed proof for the analogous statement
about Galois representations (Lemma \ref{Lem:DetectionOnSplitEll}).
\end{proof}

\vspace{4mm}
\begin{proof}[Proof of Theorem \ref{Thm:Comparison0}]
The image $W_{-1}\fh^M$ of the map $\vartheta_0$ is contained in $\Vnot P$ by
Proposition \ref{Pro:Theta0InVnotP}. By Lemma \ref{Lem:InjectViaDividing0}, the
restriction map $\Hom_\Gamma(\Vnot P, \Vnot\widetilde G) \to
\Hom_\Gamma(\im\vartheta_0, \Vnot\widetilde G)$ is injective, and we know by
Lemma \ref{Lem:GContainsAllIsogTypes} that $G$ contains all isogeny of $P$.
Hence the equality $\im\vartheta_0=\Vnot P$ must hold by Lemma
\ref{Lem:DetectionOnSplit0}.
\end{proof}


\vspace{14mm}
\section{Comparison of the motivic fundamental group with the image of
Galois}\label{Sec:CompMotivicGalois}

\begin{par}
Let $k$ be field which is finitely generated over its prime field, and let $\overline k$ be an algebraic closure of $k$. Let $\ell$ be a prime number different from the characteristic of $k$. In this section, we show that for every 1--motive $M$ over $k$ the equality $W_{-1}\fl^M = \Vell P(M)$ holds, where $\Vell P(M)$ is seen as a subspace of $\End(\Vell M)$ via the map $\alpha_\ell$ from Proposition \ref{Pro:Alpha0AlphaEll}. In analogy with the previous section we write $\Gamma$ for the absolute Galois group $\Gal(\overline k|k)$, and $\Gamma_M$ for the subgroup of $\Gamma$ consisting of those elements which act trivially on $\Tell M$. For a commutative group $C$, we introduce the notation
$$C\hotimes\IZ_\ell := \lim_{i\geq 0}C/\ell^iC$$
There is a canonical map $C\to C\hotimes \IZ_\ell$ whose kernel consists of the $\ell$--divisible elements of $C$, and we write $c\hotimes 1$ for the image of $c\in C$ under this map. 
\end{par}

\vspace{4mm}
\begin{para}\label{Par:DefinitionThetaEll}
Let $M = [u:Y\to G]$ be a 1--motive over $k$ and set $U := U(M)$ and $P := P(M)$. The action of
$\Gamma$ on $\Vell M$ is given by a group homomorphism $\rho_\ell:\Gamma \to
\GL_{\Vell M}$. The group $\Gamma_{\gr^W_\ast\!\! M}$ consisting of those
elements of $\Gamma$ which act trivially on $\gr^W_\ast\!\! \Vell M$ acts on
$\Vell M$ by unipotent automorphisms, and we have 
$$\log\rho_\ell(\gamma) = (\rho_\ell(\gamma)-1) -
\textstyle\frac12(\rho_\ell(\gamma)-1)^2\qquad  \in W_{-1}\End(\Vell M)$$
We have constructed a canonical isomorphism $\alpha_\ell: \Vell U \to
W_{-1}\End(\Vell M)$, and by composing we get a map $\vartheta_\ell :=
\alpha_\ell^{-1}\circ\log\circ\rho_\ell$. The image of $\vartheta_\ell:\Gamma_{\gr^W_\ast\!\! M}\to \Vell U$ is, except for $\ell = 2$, a $\IZ_\ell$--submodule of $\Vell U$. The $\IQ_\ell$--linear span of $\im(\vartheta_\ell)$ contains $\im(\vartheta_\ell)$ as an open subset, and is a Lie subalgebra of $\Vell U$, isomorphic via $\alpha_\ell$ to the Lie algebra $W_{-1}\fl^M$. The overall picture is similar to that in \ref{Par:DefinitionTheta0}:
$$\begin{diagram}
\setlength{\dgARROWLENGTH}{4mm}
\node[2]{\Gamma_{\gr^W_\ast\!\! M}}\arrow[2]{s,r}{\log\circ\rho_\ell}\arrow{ssw,t}{\vartheta_\ell}\arrow{e,t}{ \subseteq}\node{\Gamma}\arrow{s,r}{\rho_\ell}\\
\node[3]{\GL_{\Vell M}}\\
\node{\hspace{-18mm}\Tell P\subseteq \Vell P \subseteq \Vell U}\arrow{e,t}{\alpha_\ell}\node{W_{-1}\End(\Vell M)\hspace{-4mm}}
\end{diagram}$$
For the record, the map $\vartheta_\ell$ is explicitly given as follows: Choose
a section $s:Y\otimes \IQ_\ell \to \Vell M$ and a retraction $r:\Vell M \to
\Vell T$. Then, $\vartheta_\ell(\gamma)$ is represented by the pair 
\begin{equation}\label{Eqn:ThetaEllExplicit1}
(f-h-\textstyle\frac12e,g) \quad \in \Hom_{\IQ_\ell}(Y\otimes \IQ_\ell,\Vell
G)\times \Hom_{\IQ_\ell}(\Vell M_A, \Vell T)\tag{$\ast$}
\end{equation}
where $f,g,h$ and $e$ are given by
\begin{equation}\label{Eqn:ThetaEllExplicit2}
f(y) = \gamma s(y)-s(y) \qquad g(a) = r(\gamma \widetilde a- \widetilde a)
\qquad h = r \circ f \qquad e(y) = \gamma^2 s(y)-2\gamma s(y) +
s(y)\tag{$\ast\ast$}
\end{equation}
for all $y\in Y$ and $a \in \Vell M_A$. In the second equality, $\widetilde a$
is any element of $\Vell M$ mapping to $a\in \Vell M_A$ and $\Vell T$ is
understood to be contained in $\Vell G$. The main result of this section is the
following theorem, which in the case of a 1--motive of the form $[Y \to A]$ for
some abelian variety $A$ specialises to a Theorem of Ribet (\cite{Ribe76, Ribe79}, see
also the appendix of \cite{Hind88}). 
\end{para}

\vspace{4mm}
\begin{thm}\label{Thm:ComparisonEll}
The image of the map $\vartheta_\ell:\Gamma_{\gr^W_\ast\!\! M}\to \Vell U(M)$ is
contained and open in $\Vell P(M)$. In other words, the map $\alpha_\ell$
induces an isomorphism 
$$\Vell P(M) \xrightarrow{\:\:\cong\:\:} W_{-1}\fl^M$$
of Galois representations. Provided the condition $(\star)$ below holds for $A =
\gr^W_{-1}(M)$, the image of the map $\vartheta_\ell:\Gamma_{\gr^W_\ast\!\!
M}\to \Vell U(M)$ is equal to $\Tell P(M)$ for all but finitely many prime
numbers $\ell$.
\begin{enumerate}
\item[$(\star)$] Let $k'|k$ be a finite extension. For all but finitely many primes $\ell$ the group $H^1(L, A[\ell]\times \mu_\ell)$ is trivial, where $L$ denotes the image of $\Gal(\overline k|k')$ in $\GL(\Tell A \times \IZ_\ell(1))$. 
\end{enumerate}
\end{thm}

\begin{par}
The condition $(\star)$ holds for all abelian varieties if $k$ is of
characteristic zero. Indeed, a refinement by Serre of a theorem of Bogomolov
(\cite{SerreLetter}, theorem on p.59) guarantees (for a fixed abelian variety
$A$ over $k$) that for all but finitely many $\ell$, there exists $\gamma\in \Gal(\overline k|k)$ which acts as multiplication by a scalar $\lambda$ on $\Tell A$, such that $\lambda^2 \not \equiv 1 \bmod \ell$. If $A$ has a polarisation of degree prime to $\ell$, then $\gamma$ acts as $\lambda^2$ on $\IZ_\ell(1)$, as one can see from the Weil pairing. For such $\ell$ we have $H^n(L,\Tell A\times \IZ_\ell(1)) = 0$ for all $n\geq 0$ by Sah's Lemma, and hence $H^1(L,A[\ell] \times \mu_\ell) = 0$.
\end{par}

\vspace{4mm}
\begin{para}\label{Par:ConstructionKappaEll}
We start with the construction of a map $\kappa_\ell$ analogous to $\kappa_0$ in
the previous section. Let $K|k$ be a Galois extension contained in $\overline k$
and let $G$ be a semiabelian variety over $K$. By $H^1(K,\Tell G)$ we mean here,
and in all that follows, the group of continuous cocycles $\Gal(\overline
k|K) \to \Tell G$ for the $\ell$--adic topology on $\Tell G$, modulo
coboundaries. We construct the map
$$\kappa_\ell: G(K)\hotimes \IZ_\ell \to H^1(K,\Tell G)$$
as follows: The multiplication--by--$\ell^i$ map on $G(\overline k)$ induces a
long exact cohomology sequence, from where we can cut out the injection
$G(K)/\ell^i G(K) \to H^1(K,G[\ell^i])$. By taking limits over $i$ we get the map $\kappa_\ell$, since continuous cochain cohomology commutes with limits of compact modules.
\end{para}

\vspace{4mm}
\begin{prop}\label{Pro:KappaEllProperties}
The map $\kappa_\ell$ constructed in \ref{Par:ConstructionKappaEll} is injective
and natural in $G$ and $K$.
\end{prop}

\begin{proof}
Injectivity of $\kappa_\ell$ follows from injectivity of $G(K)/\ell^i G(K) \to
H^1(K,G[\ell^i])$ and left exactness of limits. Naturality in $G$ and $K$ is
obvious from the construction.
\end{proof}

\vspace{4mm}
\begin{prop}\label{Pro:DividingOnePointEll}
Let $M=[u: \IZ \to G]$ be a 1--motive over $k$ given by $u(1)=p \in G(k)$. The
class $\kappa_\ell(p \hotimes 1) \in H^1(\Gamma,\Tell G)$ restricts to zero in
$H^1(\Gamma_M,\Tell G)$ (continuous cochain cohomology).
\end{prop}

\begin{proof}
We must show that $\Gamma_M$ leaves all $\ell^i$--division points of $p$ fixed.
Indeed, $\Tell M/\ell^i\Tell M$ is an extension of $\IZ/\ell^i\IZ$ by
$G[\ell^i]$, and the inverse image of $1\in \IZ/\ell^i\IZ$ in $\Tell
M/\ell^i\Tell M$ is the set of $\ell^i$--division points of $p$.
\end{proof}

\vspace{4mm}
\begin{prop}\label{Pro:GroupInclusionsEll}
Let $M=[u:Y\to G]$ be a 1--motive over $k$ and consider the 1--motives 
$$M_U := [\IZ \xrightarrow{\:1\mapsto \overline u\:} U(M)] \qqet M_P := [n\IZ
\xrightarrow{\:n\mapsto n\overline u\:} P(M)]$$
where $n\geq 1$ is an integer such that the point $n\overline u$ of $U(M)$
belongs to $P(M)$ and $\overline u$ is as in Definition \ref{Def:P(M)}. The
inclusions $\Gamma_M \:\subseteq\: \Gamma_{M_U} \:\subseteq\: \Gamma_{M_P}$ hold
in $\Gamma$, and $\Gamma_M$ has finite index in $\Gamma_{M_U}$.
\end{prop}

\begin{proof}
The Galois representation $\Vell M_P$ is a subrepresentation of $\Vell M_U$, so
$ \Gamma_{M_U} \:\subseteq\: \Gamma_{M_P}$ holds trivially. As for the other
inclusion, after replacing $k$ by a finite extension over which $Y$ is constant,
even the equality $\Gamma_M \:=\: \Gamma_{M_U}$ holds. The proof consists of
recognising $\Vell M$ and $\Vell M_U$ as subquotients of products of each other,
as in the proof of \ref{Pro:GroupInclusions0}. 
\end{proof}

\vspace{4mm} 
\begin{cor}\label{Cor:EllDividing}
Let $M=[u:Y\to G]$ be a 1--motive over $k$ and let $n\geq 1$ be an integer such that the point $n\overline u$ of $U(M)$ belongs to $P(M)$. Set $\widetilde G := \gr^W_\ast\! \! G =  T \oplus A$. For every $\psi\in\Hom_k(P(M),\widetilde G)$, the cohomology class $\kappa_\ell(\psi(n\overline u)\hotimes 1)\in H^1(\Gamma, \Tell\widetilde G)$ restricts to zero in $H^1(\Gamma_M, \Tell\widetilde G)$.
\end{cor}

\begin{proof}
This is a consequence of \ref{Pro:DividingOnePointEll} and \ref{Pro:GroupInclusionsEll}, the same way \ref{Cor:HodgeDividing} was a consequence of \ref{Pro:DividingOnePoint0} and  \ref{Pro:GroupInclusions0}.
\end{proof}

\vspace{4mm}
\begin{para}\label{Par:SetupVarthetaEllOnPurePart}
We now come to the relation between $\kappa_\ell$ and $\vartheta_\ell$. Let $M =
[u:Y\to G]$ be a 1--motive over $k$ and write $M_A := M/W_{-2}M = [Y\to A]$ and
$U := U(M)$. Let 
$$\pi: U \to U_A := \cHom(Y,A) \oplus \cHom(T^\vee,A^\vee)$$ 
be the projection onto the abelian quotient $U_A$ of $U$ and let $\iota$ be the
inclusion of $\Gamma_{G \oplus M_A}$ into $\Gamma_{\gr^W_\ast\!\! M}$. We
consider the two composition maps
$$U(k)\hotimes\IZ_\ell \xrightarrow{\:\:\kappa_\ell\:\:}
H^1(\Gamma_{\gr^W_\ast\!\! M},\Tell U) \xrightarrow{\:\:(\Tell\pi)_\ast\:\:}
H^1(\Gamma_{\gr^W_\ast\!\! M},\Tell U_A) \cong \Hom(\Gamma_{\gr^W_\ast\!\!
M},\Tell U_A)$$
and
$$U(k)\hotimes\IZ_\ell \xrightarrow{\:\:\kappa_\ell\:\:}
H^1(\Gamma_{\gr^W_\ast\!\! M},\Tell U) \xrightarrow{\:\:\iota^\ast\:\:}
H^1(\Gamma_{G \oplus M_A},\Tell U) \cong \Hom(\Gamma_{G \oplus M_A},\Tell U)$$
These send $\overline u \hotimes 1$ to the homomorphisms $(\Tell\pi)\circ
\kappa_\ell(\overline u \hotimes 1)$ and $\kappa_\ell(\overline u \hotimes
1)\circ \iota$ respectively. Here we have used that $\Gamma_{\gr^W_\ast M}$ acts
trivially on $\Tell U_A$ and that $\Gamma_{G \oplus M_A}$ acts trivially on
$\Tell U$.
\end{para}

\vspace{4mm}
\begin{lem}\label{Lem:VarthetaEllOnPurePart}
Notations being as in \ref{Par:SetupVarthetaEllOnPurePart}, the equalities
$$(\Tell\pi)\circ \vartheta_\ell = (\Tell\pi)\circ \kappa_\ell(\overline u
\hotimes 1) \qqet \vartheta_\ell \circ \iota = \kappa_\ell(\overline u \hotimes
1)\circ \iota$$
hold in $\Hom(\Gamma_{\gr^W_\ast\!\! M},\Tell U_A)$ and in $\Hom(\Gamma_{G
\oplus M_A},\Tell U)$ respectively.
\end{lem}

\begin{proof}
Let us choose an $\ell$--division sequence of $\overline u$, which we may
represent by two sequences of points
$$(u_i)_{i=0}^\infty \quad\mbox{in}\:\:\Hom_{\overline k}(Y,G) \qqet
(\eta_i)_{i=0}^\infty \quad\mbox{in}\:\:\Ext_{\overline k}^1(M_A,T)$$
with $u_0=u$ and $\ell u_i=u_{i-1}$, and with $\eta_0 = \eta$ and $\ell \eta_i
=\eta_{i-1}$. Here $\eta \in \Ext_k^1(M_A,T)$ is the class given by
$M$. The $\eta_i$'s define extensions $G_i$ of $A$ by $T$ together with maps
$m_i:G\to G_i$. With the help of these division sequences we construct a section
$s:Y\otimes \IZ_\ell \to \Tell M$ and a retraction $r:\Tell M \to \Tell T$ as
follows: 
$$s((y_i)_{i=0}^\infty) = (y_i,u_i(y_i))_{i=0}^\infty \qqet
r(y_i,x_i)_{i=0}^\infty = m_i(x_i)_{i=0}^\infty$$
Using the section $s$ and the retraction $r$ we can write down the map
$\vartheta_\ell$ as in \ref{Par:DefinitionThetaEll}, equations  ($\ast$) and
($\ast\ast$). The remainder of the proof of \ref{Lem:VarthetaEllOnPurePart} is
then litterally the same as the proof of \ref{Lem:Vartheta0OnPurePart}.
\end{proof}

\vspace{4mm}
\begin{lem}\label{Lem:SSDetectionSubEll}
Let $0\to T \to G  \xrightarrow{\:\:\pi\:\:} A \to 0$ be a semiabelian variety over $k$, and let $G'$ be a semiabelian subvariety of $G$. A Galois invariant $\IZ_\ell$--submodule  $X$ of $\Tell G$ is contained in $\Tell G'$ if and only if the inclusions $\pi(X) \subseteq \Tell\pi(G')$ and $X\cap \Tell T \subseteq \Tell G'$ hold.
\end{lem}

\begin{proof}
The proof is analogous to the proof of \ref{Lem:SSDetectionSub0}. That there are no nonzero Galois equivariant maps $\pi(X)\to \Tell(T/(T\cap G'))$ can be seen for example by looking at absolute values of eigenvalues of Frobenius elements.
\end{proof}

\vspace{4mm}
\begin{prop}\label{Pro:ThetaEllInVellP}
The image of the map $\vartheta_\ell:\Gamma_{\gr^W_\ast\!\! M}\to \Vell U$ is contained in $\Vell P$, and even in $\Tell P$ for all but finitely many $\ell \neq \mathrm{char}(k)$.
\end{prop}

\begin{proof}
This follows from \ref{Lem:VarthetaEllOnPurePart}, \ref{Lem:SSDetectionSubEll} and naturality of the map $\kappa_\ell$, the same way \ref{Pro:Theta0InVnotP} follows from \ref{Lem:Vartheta0OnPurePart}, \ref{Lem:SSDetectionSub0} and naturality of the map $\kappa_0$. Indeed, if $n\geq 1$ is an integer such that the point $n\overline u$ of $U$ belongs to $P$, then Lemma \ref{Lem:VarthetaEllOnPurePart} tells us that we have 
$$(\Tell\pi)\circ \vartheta_\ell = n^{-1}\cdot(\Tell\pi)\circ\kappa_\ell(n\overline u \hotimes 1) \qqet \vartheta_\ell \circ\iota = n^{-1}\cdot\kappa_\ell(n\overline u \hotimes 1)\circ\iota$$
hence $\im(\vartheta_\ell) \subseteq \Tell P$ holds by  \ref{Lem:SSDetectionSubEll}  as soon as $\ell$ is odd and does not divide $n$ (mind that $\im(\vartheta_2)$ is in general not a $\IZ_2$--module). 
\end{proof}

\vspace{4mm}
\begin{lem}\label{Lem:HSSSEll}
Let $M=[Y\to G]$ be a 1--motive over $k$ and set $\widetilde G := \gr^W_\ast G = T \oplus A$.  Of the following statements, (a) holds for all, and (b) for all but finitely many primes $\ell \neq \mathrm{char}(k)$
\begin{enumerate}
 \item[(a)] The map $H^1(\fl^M,\Vell \widetilde G)\to \Hom_{\fl^M}(W_{-1}\fl^M,\Vell\widetilde G)$ that sends the class of a cocycle $c$ to the restriction of $c$ to $W_{-1}\fl^M$ is an isomorphism.
 \item[(b)] Suppose that the condition $(\star)$ in Theorem \ref{Thm:ComparisonEll} holds for the abelian variety $A$. Let $L^M$ be the image of $\Gamma$ in $\GL(\Tell M)$ and write $W_{-1}L^M$ for the subgroup of those elments acting trivially on $\Tell(\gr^W_\ast\!\! M)$. The map $H^1(L^M, \widetilde G[\ell])\to \Hom_\Gamma(W_{-1}L^M,\widetilde G[\ell])$ that sends the class of a cocycle $c$ to the restriction of $c$ to $W_{-1}L^M$ is injective.
\end{enumerate}
\end{lem}

\begin{proof}
The Lie subalgebra $W_{-1}\fl^M$ of $\fl^M$ consist of those elements which act trivially on $\Vell \widetilde G$. In low degrees, the Hochschild--Serre spectral sequences associated with the Lie algebra extension $0\to W_{-1}\fl^M\to \fl^M \to \fl^{\widetilde G} \to 0$ yields the exact sequence
$$H^1(\fl^{\widetilde G},\Vell\widetilde G) \to H^1(\fl^M,\Vell\widetilde G) \xrightarrow{\:\:\ast\:\:} H^1(W_{-1}\fl^M,\Vell\widetilde G)^{\fl^{\widetilde G}}\to H^2(\fl^{\widetilde G},\Vell\widetilde G)$$
We can identify $H^1(W_{-1}\fl^M,\Vell\widetilde G)$ with $\Hom_{\IQ_\ell}(W_{-1}\fl^M,\Vell\widetilde G)$, and under this identification the map $(\ast)$ is given by restricting cocycles as in the statement of the lemma. Hence, to finish the proof it suffices to show that the first and last term in the above sequence is trivial. Indeed, it follows from an adaptation of Serre's vanishing criterion given in \cite{Serre2} that $H^i(\fl^{\widetilde G},\Vell \widetilde G)$ is zero for all $i\geq 0$. This shows (a). For (b), let us choose a finite Galois extension $k'$ of $k$ such that $T$ is split and such that $Y$ is constant over $k'$. Denote by $(L^{\widetilde G})'$ and $(L^M)'$ the images of $\Gal(\overline k|k')$ in $\GL(\Tell \widetilde G)$ and $\GL(\Tell M)$ respectively. The subgroup $W_{-1}(L^M)'$ of $(L^M)'$ consist then of those elements which act trivially on $\Tell \widetilde G$. As in the Lie algebra case, we obtain an exact sequence 
$$H^1((L^{\widetilde G})', \widetilde G[\ell]) \to H^1((L^M)', \widetilde G[\ell]) \to H^1(W_{-1}(L^M)',\widetilde G[\ell])^{(L^M)'}\to H^2((L^{\widetilde G})', \widetilde G[\ell])$$
upon which $\Gal(k'|k)$ acts. By $(\star)$, the first group in this sequence is trivial for almost all $\ell$, and it follows from Maschke's theorem that if moreover $\ell$ does not divide the order of the finite group $\Gal(k'|k)$, then the sequence obtained by taking invariants is exact, and reads
$$0 \to H^1(L^M, \widetilde G[\ell]) \to H^1(W_{-1}L^M,\widetilde G[\ell])^\Gamma\to H^2(L^{\widetilde G}, \widetilde G[\ell])$$
where the first map is given by restriction of cocycles, hence the claim.
\end{proof}

\vspace{4mm}
\begin{lem}\label{Lem:InjectViaDividingEll}
Let $M=[Y\to G]$ be a 1--motive over $k$ and set $\widetilde G := \gr^W_\ast G = T \oplus A$ and $P := P(M)$. Of the following statements, (a) holds for all, and (b) for all but finitely many primes $\ell \neq \mathrm{char}(k)$:
\begin{enumerate}
\item[(a)] The map $\alpha_\ell^\ast: \Hom_\Gamma(\Vell P, \Vell \widetilde G) \to \Hom_\Gamma(W_{-1}\fl^M, \Vell\widetilde G)$ given by $\alpha_\ell^\ast(f) = f\circ\alpha_\ell^{-1}$ is injective. 
\item[(b)] Suppose that the condition $(\star)$ in Theorem \ref{Thm:ComparisonEll} holds for the abelian variety $A$. The map $\Hom_\Gamma(\Tell P, \widetilde G[\ell]) \to \Hom_\Gamma(\im(\vartheta_\ell), \widetilde G[\ell])$ sending a homomorphism to its restriction to $\im(\vartheta_\ell)$, is well defined and injective.
\end{enumerate}
\end{lem}

\begin{proof}
As in the Hodge situation, will construct an auxiliary, injective map, and prove in a second step that this map is equal to $\alpha_\ell^\ast$, respectively the restriction. Write $L^M$ and $W_{-1}L^M$ for the image of $\Gamma$, respectively of $\Gamma_{\gr^W_\ast\!M}$ in $\GL(\Tell M)$. We start with (a), using this commutative diagram:
$$\begin{diagram}
\setlength{\dgARROWLENGTH}{4mm}
\node[3]{\Hom_k(P,\widetilde G)\otimes\IQ_\ell}\arrow{s,r}{(\ast)}\arrow{sw,--}\\
\node{0}\arrow{e}\node{H^1(L^M, \Vell \widetilde G)}\arrow{e}\arrow{s,l}{\cong}\node{H^1(k,\Vell \widetilde G)}\arrow{e,t}{\res}\node{H^1(k_M,\Vell \widetilde G)}\\
\node[2]{\Hom_\Gamma(W_{-1}\fl^M,\Vell\widetilde G)}
\end{diagram}$$
The map $(\ast)$ sends a homomorphism $\psi$ to the element $n^{-1}\kappa_\ell(\psi(nu) \hotimes 1)$, where $n\geq 1$ is an integer such that $n\overline u$ belongs to $P$. This map is injective by minimality of $P$ and injectivity of the map $\kappa_\ell$, and its composite with the restriction is zero by Corollary \ref{Cor:EllDividing}, hence the dashed arrow. The lower vertical isomorphism is obtained by taking $\Gamma$--fixed
points of the isomorphism of Lemma \ref{Lem:HSSSEll} and taking into account
that $H^1(L^M,\Vell \widetilde G) = H^1(\fl^M,\Vell \widetilde G)^\Gamma$ by
Lazard's theorem comparing Lie group cohomology with Lie algebra cohomology
(\cite{Lazard}, V.2.4.10). As in the proof of \ref{Lem:InjectViaDividing0} it follows from Lemma \ref{Lem:VarthetaEllOnPurePart} that the so obtained injection
$$\Hom_\Gamma(\Vell P, \Vell\widetilde G)\cong\Hom_k(P,G)\otimes\IQ_\ell  \to H^1(L^M,\Vell \widetilde G)$$
is equal to $\alpha_\ell^\ast$, hence $\alpha_\ell^\ast$ is injective. The modifications for part (b) are as follows: For those $\ell$ which do not divide $n$, we consider the diagram
$$\begin{diagram}
\setlength{\dgARROWLENGTH}{4mm}
\node[3]{\Hom_k(P,\widetilde G)\otimes\IZ/\ell\IZ}\arrow{s,r}{(\ast)}\arrow{sw,--}\\
\node{0}\arrow{e}\node{H^1(L^M, \widetilde G[\ell])}\arrow{e}\arrow{s}\node{H^1(k, \widetilde G[\ell])}\arrow{e,t}{\res}\node{H^1(k_M, \widetilde G[\ell])}\\
\node[2]{\Hom_\Gamma(\im(\vartheta_\ell),\widetilde G[\ell])}
\end{diagram}$$
where ($\ast$) sends $\psi\otimes 1$ to the class of the cocycle $\sigma \mapsto n^{-1}(\sigma\psi(u_1)-\psi(u_1))$,  and $u_1\in P(\overline k)$ is any point such that $\ell u_1 = n\overline u$. It follows essentially from the Mordell--Weil Theorem and Dirichlet's Unit Theorem that the group $\widetilde G(k)$ is isomorphic to a direct sum of a free group and a finite group. The injective map $\Hom_k(P,\widetilde G) \to \widetilde G(k)$ sending $\psi$ to $\psi(n\overline u)$ remains thus injective after tensoring with $\IZ/\ell\IZ$ for all but finitely many $\ell$, and hence ($\ast$) is injective for all but finitely many $\ell$. The vertical map is given by sending the class of a cocycle $c$ to the composition 
$$\im(\vartheta_\ell) \xrightarrow{\:\exp\:} W_{-1}L^M \xrightarrow{\:\subseteq\:} L^M \xrightarrow{\:\:c\:\:}\widetilde G[\ell]$$ 
which is a group homomorphism as long as $\ell\neq 2$. By Lemma \ref{Lem:HSSSEll} and because the exponential is bijective, the vertical map is injective for all but finitely many $\ell$. For almost all $\ell$, the map $\Hom(P,\widetilde G)\otimes\IZ/\ell\IZ \to \Hom_\Gamma(\Tell P, G[\ell])$ is an isomorphism, and we find that the composite map
$$\Hom_\Gamma(\Tell P,\widetilde G[\ell]) \xrightarrow{\:\:\cong\:\:} \Hom_k(P,\widetilde G)\otimes\IZ/\ell\IZ \to H^1(L^M,\widetilde G[\ell]) \to \Hom_\Gamma(\im(\vartheta_\ell),\widetilde G[\ell])$$
is well defined and injective for almost all $\ell$. That this map is the restriction is again a direct computation using the equalities from Lemma \ref{Lem:VarthetaEllOnPurePart} modulo $\ell$.
\end{proof}

\vspace{4mm}
\begin{lem}\label{Lem:DetectionOnSplitEll}
Let $G$ be a semiabelian variety and let $\widetilde G$ be a split semiabelian variety over $k$ containing (over $\overline k$) all isogeny types of $G$ (Definition \ref{Def:AllIsogenyTypes}). There exists an open subgroup $\Gamma'$ of $\Gamma$ such that of the following statements, (a) holds for all, and (b) for all but finitely many primes $\ell$: 
\begin{enumerate}
\item[(a)] If the restriction map $\Hom_{\Gamma'}(\Vell G, \Vell \widetilde G) \to \Hom_{\Gamma'}(V, \Vell \widetilde G)$ is injective for a Galois invariant $\IQ_\ell$--linear subspace $V$ of $\Vell G$, then $V$ is equal to $\Vell G$.
\item[(b)] If the restriction map $\Hom_{\Gamma'}(G[\ell], \widetilde G[\ell]) \to \Hom_{\Gamma'}(V, \widetilde G[\ell])$ is injective for a Galois invariant subgroup $V$ of $G[\ell]$, then $V$ is equal to $G[\ell]$.
\end{enumerate}
\end{lem}

\begin{proof}
Write $G$ as an extension of an abelian variety $A$ by a torus $T$, and $\widetilde G$ as a sum of an abelian variety $\widetilde A$ and a torus $\widetilde T$. We now replace $k$ by a finite Galois extension such that there exists an integer $n$ and morphisms $A \to \widetilde A^n$ and $T\to\widetilde T^n$ with finite kernels defined over $k$, and show that the statements hold for $\Gamma' = \Gamma$. Consider the commutative diagram with exact rows:
$$\begin{diagram}
\setlength{\dgARROWLENGTH}{4mm}
\node{0}\arrow{e}\node{V\cap \Vell
T}\arrow{s,l}{\subseteq}\arrow{e}\node{V}\arrow{s,l}{\subseteq}\arrow{e}\node{
\pi V}\arrow{s,l}{\subseteq}\arrow{e}\node{0}\\
\node{0}\arrow{e}\node{\Vell T}\arrow{e}\node{\Vell
G}\arrow{e,t}{\pi}\node{\Vell A}\arrow{e}\node{0}
\end{diagram}$$
We apply $\Hom_\Gamma(-,\Vell \widetilde A)$ to this diagram and get
$$\begin{diagram}
\setlength{\dgARROWLENGTH}{4mm}
\node{0}\arrow{e}\node{\Hom_\Gamma(\pi V, \Vell \widetilde
A)}\arrow{e,t}{\cong}\node{\Hom_\Gamma(V,\Vell\widetilde A)}\arrow{e}\node{0}\\
\node{0}\arrow{e}\node{\Hom_\Gamma(\Vell A,\Vell\widetilde
A)}\arrow{n}\arrow{e,t}{\cong}\node{\Hom_\Gamma(\Vell G, \Vell\widetilde
A)}\arrow{n}\arrow{e}\node{0}
\end{diagram}$$
using that $\Hom(\Vell T,\Vell \widetilde A)=0$ and $\Hom(V\cap \Vell T,\Vell \widetilde A)=0$ for weight reasons. The vertical maps are injective
by hypothesis. The Galois representations $\Vell A$ and $\Vell \widetilde A$ are
semisimple as Galois modules (by Faltings in characteristic zero, and by Tate,
Zahrin and Mori in positive characteristic), and all simple factors appearing in
$\Vell A$ also appear in $\Vell \widetilde A$ by hypothesis. Hence injectivity
of the left hand vertical map implies the equality $\pi V = \Vell A$. Next, we
apply $\Hom_\Gamma(-, \Vell\widetilde T)$ instead, and find
$$\begin{diagram}
\setlength{\dgARROWLENGTH}{4mm}
\node{0}\arrow{e}\node{\Hom_\Gamma(V,\Vell\widetilde
T)}\arrow{e}\node{\Hom_\Gamma(V\cap \Vell T,\Vell\widetilde T)}\arrow{e}
\node{\Ext^1_\Gamma(\pi V,\Vell\widetilde T)}\\
\node{0}\arrow{e}\node{\Hom_\Gamma(\Vell G, \Vell\widetilde
T)}\arrow{n}\arrow{e}\node{\Hom_\Gamma(\Vell T, \Vell\widetilde
T)}\arrow{n}\arrow{e}\node{\Ext^1_\Gamma(\Vell A,\Vell\widetilde T)}\arrow{n,=}
\end{diagram}$$
using that $\Hom(\Vell A, \Vell\widetilde T)=0$ for weight reasons. The right hand side vertical map is injective by hypothesis, and we have already shown that the left hand vertical map is the identity. The middle vertical map is therefore injective. The Galois representation $\Vell T$ is semisimple because $\Vell T \otimes\IQ_\ell(-1)$ is so by Maschke's Theorem, hence the middle vertical map can only be injective if the equality $V\cap \Vell T = \Vell T$ holds, and so we are done for (a). The modular case follows along the same lines.
\end{proof}

\vspace{4mm}
\begin{proof}[Proof of Theorem \ref{Thm:ComparisonEll}]
The image of the map $\vartheta_\ell$ is contained in $\Vell P$ for all, and even in $\Tell P$ for almost all $\ell \neq \mathrm{char}(k)$ by Proposition \ref{Pro:ThetaEllInVellP}. In particular we have $W_{-1}\fl^M \subseteq \Vell P$. With the notations from Lemma \ref{Lem:InjectViaDividingEll}, said lemma tells us that the restriction maps 
$$\Hom_\Gamma(\Vell P, \Vell\widetilde G) \to \Hom_\Gamma(W_{-1}\fl^M, \Vell\widetilde G) \qqet \Hom_\Gamma(\Tell P, \widetilde G[\ell]) \to \Hom_\Gamma(\im\vartheta_\ell, \widetilde G[\ell])$$
are well defined and injective for all, respectively almost all $\ell$, provided the condition ($\star$) holds. By Lemma \ref{Lem:GContainsAllIsogTypes}, $\widetilde G$ contains all isogeny types of $P$. Hence the equality $W_{-1}\fl^M = \Vell P$ must hold for all, and $\im\vartheta_\ell=\Tell P$ must hold for almost all $\ell$ by Lemma \ref{Lem:DetectionOnSplitEll}.
\end{proof}


\vspace{14mm}
\section{Compatibility of the comparison isomorphisms}\label{Sec:Compatibility}

\begin{par}
In this section we prove Theorems 1 and 2 as stated in the introduction, and
also show that the comparison isomorphisms between realisations of $M$ and of
$P(M)$ are compatible. We start with the proof of Theorem 2, where we only have
to assemble results:
\end{par}

\vspace{4mm}
\begin{proof}[Proof of Theorem 2]
For every 1--motive $M$ over a field $k$ we have constructed a semiabelian variety $P(M)$ in \ref{Def:P(M)}, and the construction is compatible with base change as required in statement (1) of the theorem. Statements (2) and (3) are the contents of Theorems \ref{Thm:Comparison0} and \ref{Thm:ComparisonEll} respectively.
\end{proof}

\vspace{4mm}
\begin{para}
We now come to the compatibility problem: Let $M$ be a 1--motive over a finitely generated field $k$ of characteristic zero, choose a complex embedding $k \to \IC$ and a prime number $\ell$, and
denote by $\overline k$ the algebraic closure of $k$ in $\IC$. The comparison
isomorphism $\Vnot M \otimes \IQ_\ell \to \Vell M $ induces an isomorphism
$$\End_\IQ(\Vnot M) \otimes \IQ_\ell \xrightarrow{\:\:\cong\:\:}
\End_{\IQ_\ell}(\Vell M)$$
and we have identified the Lie algebra $\fl^M $ with a subalgebra of $\fh^M
\otimes \IQ_\ell$ via this isomorphism. In sections \ref{Sec:CompMotivicHodge}
and \ref{Sec:CompMotivicGalois} we constructed canonical maps
$$W_{-1}\fh^M \to \Vnot P(M) \qqet W_{-1}\fl^M \to \Vell P(M)$$
and proved that they are isomorphisms. The object $P(M)$ is a semiabelian
variety, and we also have the comparison isomorphism $\Vnot P(M) \otimes
\IQ_\ell \to \Vell P(M)$. Putting things together, we get the following diagram
\begin{equation}\label{Eqn:CompatibilityDgr}
\begin{diagram}
\node[2]{\End_\IQ(\Vnot M) \otimes
\IQ_\ell}\arrow{e,t}{\cong}\node{\End_{\IQ_\ell}(\Vell M)}\\
\node{W_{-1}\fh^M \otimes
\IQ_\ell}\arrow{ne,t}{\subseteq}\arrow{se,tb}{\mathrm{Theorem\:
\ref{Thm:Comparison0}}} {\cong} \node[3]{W_{-1}\fl^M}\arrow{nw,t} {\subseteq}
\arrow{sw,tb} {\mathrm{Theorem\: \ref{Thm:ComparisonEll}}}{\cong}\\
\node[2]{\Vnot P(M)\otimes\IQ_\ell}\arrow{e,t}{\cong}\node{\Vell P(M)}
\end{diagram}\tag{$\ast$}
\end{equation}
We next check:
\end{para}

\vspace{4mm}
\begin{prop}\label{Pro:Compatibility}
The diagram (\ref{Eqn:CompatibilityDgr}) commutes.
\end{prop}

\begin{proof}
We consider the group scheme $\Gamma^{\Hdg}_{\gr^\ast\! M}$ over $\spec\IQ$ and
the profinite group $\Gamma^{\Gal}_{\gr^\ast\! M}$, given by
$$\Gamma^{\Hdg}_{\gr^\ast\! M} = \{\gamma\in \Gamma^{\Hdg} \tq
\gamma|_{\Vnot(\gr^\ast\! M)}=\id\} \qqet \Gamma^{\Gal}_{\gr^\ast\! M} =
\{\gamma\in\Gamma^{\Gal}\tq \gamma|_{\Vell(\gr^\ast\! M)}=\id\}$$
respectively, where $\Gamma^{\Hdg}$ denotes the absolute Hodge group, and
$\Gamma^{\Gal}$ the absolute Galois group of $k$. In the following diagram, the
triangles on the left and on the right commute by definition of $\vartheta_0$
and $\vartheta_\ell$:
$$\begin{diagram}
\setlength{\dgARROWLENGTH}{5mm}
\node[2]{\End_\IQ(\Vnot M) \otimes\IQ_\ell}
\arrow{e,t}{\cong}\node{\End_{\IQ_\ell}(\Vell M)}\\
\node{\Gamma^{\Hdg}_{\gr^\ast\!
M}}\arrow{e,t}{\log\circ\rho_0}\arrow{se,b}{\vartheta_0}\node{W_{-1}
\End_\IQ(\Vnot M) \otimes\IQ_\ell} \arrow{n,l}{\subseteq}\arrow{e,t}{\cong}
\node{W_{-1}\End_{\IQ_\ell}(\Vell M)}\arrow{n,l}{\subseteq} \node{
\Gamma^{\Gal}_{\gr^\ast\!
M}}\arrow{w,t}{\log\circ\rho_\ell}\arrow{sw,b}{\vartheta_\ell}\\
\node[2]{\Vnot U(M)\otimes\IQ_\ell}\arrow{n,lr}{\alpha_0\otimes
1}{\cong}\arrow{e,t}{\cong} \node{\Vell U(M)}\arrow{n,lr}{\alpha_\ell}{\cong}\\
\node[2]{\Vnot P(M)\otimes\IQ_\ell}\arrow{n,l}{\subseteq}\arrow{e,t}{\cong}
\node{\Vell P(M)}\arrow{n,l}{\subseteq}
\end{diagram}$$
By definition, $W_{-1}\fh^M\otimes\IQ_\ell$ is the $\IQ_\ell$--linear span of
$\im(\log\circ\rho_0)$, and $W_{-1}\fl^M$ is the $\IQ_\ell$--linear span of
$\im(\log\circ\rho_\ell)$. The top and the bottom square in the above diagram
commute by naturality of the comparison isomorphisms, and all that remains is to
show that the central square commutes. By construction of the semiabelian
variety $U(M)$, see \ref{Par:ConstructionU(M)}, and naturality of comparison
isomorphisms, this follows from the commutativity of the two squares pictured in
\ref{Rem:HomExtAndComparison}.
\end{proof}

\vspace{4mm}
\begin{proof}[Proof of Theorem 1]
In the situation of Theorem 1, we know by Theorem \ref{Thm:ElladicInMumfordTate} that $\fl^M$ is contained
in $\fh^M \otimes \IQ_\ell$ once we identify $\End(\Vell M)$ with $\End(\Vnot M) \otimes \IQ_\ell$ via the
comparison isomorphism $\Vell M \cong \Vnot M \otimes \IQ_\ell$. By Proposition
\ref{Pro:Compatibility}, the inclusion $\fl^M  \subseteq \fh^M \otimes \IQ_\ell$
is compatible with the comparison isomorphism $\Vell P(M) \cong \Vnot(M)\otimes
\IQ_\ell$, hence an equality (for dimension reasons, this would also follow
without referring to \ref{Pro:Compatibility}).
\end{proof}


\vspace{14mm}
\section{Corollaries}\label{Sec:Corollaries}

\begin{par}
We propose in this last section to illustrate how our theorems can be used to
aggress questions about the geometry and arithmetic of semiabelian varieties. We
use them here to continue some work started by Ribet and Jacquinot in
\cite{Jacq87} about so--called \emph{deficient points} on semiabelian varieties.
We stick here to semiabelian varieties over number fields. Motivated by
\cite{Bert11}, it would be equally interesting to consider semiabelian varieties
defined over smooth curves over $\IC$.
\end{par}

\vspace{4mm}
\begin{para}\label{Par:DeficientSetup}
Let $k$ be a number field with algebraic closure $\overline k$, let $M=[Y
\xrightarrow{\:\:u\:\:} G]$ be a 1--motive over $k$, and denote by $k_M
\subseteq \overline k$ the fixed field of the pointwise stabiliser of $\Tell M$
in $\Gal(\overline k|k)$. Equivalently, $k_M$ is the smallest subfield of
$\overline k$ over which $Y$ is constant and all $\ell$--division points of
$u(Y)$ are defined.
\end{para}

\vspace{4mm}
\begin{defn}\label{Def:DeficientPoint}
A point $Q\in G(k)$ is called \emph{deficient} if it is $\ell$--divisible in the
group $G(k_M)$. We write $D_M(k)$ for the subgroup of $G(k)$ of deficient
points. 
\end{defn}

\vspace{4mm}
\begin{para}
Our goal is to describe the group $D_M(k)$, which has been studied in
\cite{Jacq87} in the case of a 1--motive of the form $[0\to G]$ where $G$ is an
extension of an abelian variety by $\IG_m$. We will show here that it is
independent of the prime $\ell$, finitely generated of rank $\leq r$ where $r$
only depends on the dimension of $\Vell M$, say. To do so, we will give a
geometrical construction of $D_M(k)$, roughly in the following way: Recall that
the semiabelian variety $P := P(M)$ associated with $M$ in \ref{Def:P(M)} is
supposed to be a Lie algebra object, acting on $G$. So we can, in a sense yet to
be clarified, consider derivations of $P$ with values in $G$. The deficient
points will, up to multiplying by integers, be the images of $n\overline u$
under such derivations, were $n\overline u$ is a multiple of the rational
$\overline u$ point on $U(M)$ defined by $M$, see \ref{Par:ConstructionU(M)}.
That derivations play a key role in the construction of deficient points is
already visible in \cite{Jacq87}, the construction there being attributed to
Breen. We first state an immediate corollary to Theorem \ref{Thm:ComparisonEll},
which shows that the notion of deficiency does not depend on the prime $\ell$.
This corollary can also be utilised to extend the definition of deficient points
to 1--motives over an arbitrary base.
\end{para}

\vspace{4mm}
\begin{cor}[To Theorem \ref{Thm:ComparisonEll}]\label{Cor:DeficientGeometric}
Let $M=[Y\xrightarrow{\:\:u\:\:}G]$ be a 1--motive over $k$, let $Q\in G(k)$ be
a rational point and define $M^+ = [Y\oplus\IZ \xrightarrow{\:\:u^+\:\:}G]$ by
$u^+(y,n)=u(y)+nQ$. The following are equivalent:
\begin{enumerate}
 \item The point $Q\in G(k)$ is deficient for one (or for all) primes $\ell$.
 \item The map $P(M^+)\to P(M)$ induced by the canonical morphism $M\to M^+$ is
an isogeny.
\end{enumerate}
\end{cor}

\begin{proof}
The morphisms of 1--motives $M\to M^+ \to [\IZ\to 0]$ induce a short exact
sequence of $\ell$--adic representations $0\to \Vell M \to \Vell M^+ \to
\IQ_\ell \to 0$. Define fields $k_M$ and $k_{M^+}$ as in
\ref{Par:DeficientSetup}, so $k_{M^+}$ is a Galois extension of $k_M$. The
Galois group $\Gal(k_{M^+}|k_M)$ identifies canonically with a compact subgroup
of $\Vell M$, hence is commutative and has the structure of a finitely
generated, free $\IZ_\ell$--module.\\
The point $Q \in G(k)$ is deficient if and only the field extension
$k_{M^+}|k_M$ is trivial. This in turn is the case if and only if
$\Gal(k_{M^+}|k_M)$ is trivial, or, yet in other words, if the Lie algebra
morphism $\fl^{M^+}\to \fl^M$ is an isomorphism. The graded quotients of weight
$0$ of $\fl^{M^+}$ and $\fl^M$ are the same, hence, by Theorem
\ref{Thm:ComparisonEll}, the map $\fl^{M^+}\to \fl^M$ is an isomorphism
precisely if $\Vell P(M^+)\to \Vell P(M)$ is an isomorphism.
\end{proof}

\vspace{4mm}
\begin{para}\label{Par:DerivationsDef}
We now explain what we mean with derivations of $P$ with values in $G$.  Let $P$
be a semiabelian variety, extension of an abelian variety $A_P$ by a torus
$T_P$, equipped with a Lie structure $\lambda$, and let $G$ be a semiabelian
variety, extension of $A$ by $T$, equipped with an  action $\alpha$ of $P$. The
Lie structure and the action are given by maps of Galois modules
$$\lambda: T_P^\vee \to \Hom_{\overline k}(A_P,A_P^\vee) \qqet \alpha:T^\vee \to
\Hom_{\overline k}(A_P,A^\vee)$$
as we have explained in \ref{Par:WhatIsLieBracket}. Along the same lines as in
\ref{Par:WhatIsLieBracket}, we figure that a
\emph{derivation}\footnote{actually: a $\pi_1^\upmot(M)$--equivariant
derivation} of $P$ into $G$ corresponds to a pair of morphisms of $\overline
k$--group schemes $\partial = (\partial_T,\partial_A)$ from $T_P$ to $T$ and
from $A_P$ to $A$ respectively, such that the Leibniz rule 
$$(\partial_A^\vee\circ\alpha(\chi))+(\alpha(\chi)^\vee \circ \partial_A)  =
\lambda(\partial^\vee_T(\chi))$$
holds for all $\chi\in T^\vee$. This is an equality in $\Hom_{\overline
k}(A_P,A_P^\vee)$. We denote by $\Der_{\overline k}(P,G)$ the group of all such
derivations from $P$ to $G$. It is a subgroup of $\Hom_{\overline
k}(T_P,T)\times \Hom_{\overline k}(A_P,A)$, hence it is finitely generated and
free, and comes equipped with an action of $\Gal(\overline k|k)$. 
\end{para}

\vspace{4mm}
\begin{para}\label{Par:DerivationsSetup}
Back to our concrete situation, let $M=[u:Y\to G]$ be a 1--motive, where $G$ is
an extension of $A$ by $T$. Let $P = P(M)$ be the semiabelian variety defined in
\ref{Def:P(M)}, and write $T_P$ for the torus part and $A_P$ for the abelian
quotient of $P$. Recall that $A_P$ is an abelian subvariety of $\cHom(Y,A)\times
\cHom(T^\vee,A^\vee)$ and $T_P$ a subtorus of $\cHom(Y,T)$. We identify
$A_P^\vee$ with a quotient of $(Y \otimes A^\vee) \times (T^\vee \otimes A)$ and
$T_P^\vee$ with a quotient of $Y\otimes T^\vee$. The dual of $P$ is the
morphism 
$$w:T_P^\vee \to A_P^\vee \qquad w(y\otimes\chi) = w(y\otimes \chi) = (y\otimes
v^\vee(\chi), \chi\otimes v(y))$$ 
where $v:Y\to A$ is the composite of $u$ with the projection $G\to A$, and
$v^\vee:T^\vee\to A^\vee$ is the corresponding map in the 1--motive dual to $M$.
The maps $\lambda$, defining the Lie algebra structure of $P = P(M)$, and
$\alpha$, defining the action of $P$ on $G$ are given by
$$\lambda:Y\otimes T^\vee \to \Hom(A_P,A_P^\vee) \qquad \lambda(y\otimes
\chi)(f,g) = (y\otimes g(\chi), \chi\otimes f(y))$$
and
$$\alpha: T^\vee \to \Hom(A_P,A^\vee) \qquad \alpha(\chi)(f,g) = g(\chi)$$
respectively. Observe that in the special case where $Y$ is trivial and $G$ a
non--isotrivial extension of a simple abelian variety $A$ by $\IG_m$, we have
$P(M) = \cHom(\IZ,A^\vee) = A^\vee$, and $\Der_k(P(M),A)$ consists of
homomorphisms $\partial_A \in \Hom(A^\vee, A)$ with the property $\partial_A +
\partial_A^\vee=0$. So the group of derivations is isomorphic to the quotient of
$\Hom(A^\vee,A)$ modulo the N\'eron--Severi group of $A$.
\end{para}

\vspace{4mm}
\begin{lem}\label{Lem:TheseAreDerivations}
Let $y^+$ be an element of $Y$. The pair of homomorphisms
$(\partial_T,\partial_A)$ given by
$$\partial_A:A_P \xrightarrow{\:\:\subseteq\:\:} A_U
\xrightarrow{\:\:(f,g)\mapsto f(y^+)\:\:}A \qqet \partial_T:T_P
\xrightarrow{\:\:\subseteq\:\:} T_U \xrightarrow{\:\:h\mapsto h(y^+)\:\:}T$$
is a derivation of $P$ to $G$.
\end{lem}

\begin{proof}
The duals of $\partial_A$ and $\partial_P$ are given by $\partial_A^\vee: a
\mapsto (y^+\otimes a,0)$ and $\partial_T^\vee: \chi \mapsto y^+ \otimes \chi$.
For $(f,g) \in A_P$, we compute:
\begin{multline*}
\partial_A^\vee(\alpha(\chi)(f, g))+\alpha(\chi)^\vee(\partial_A(f,g)) =
\partial_A^\vee(g(\chi)) + \alpha(\chi)^\vee(f(y^+)) = \\
= (y^+\otimes g(\chi),0) + (0,\chi\otimes f(y^+)) = (y^+\otimes g(\chi),  \chi
\otimes f(y^+))  = \\
= \lambda(y^+\otimes \chi)(f,g) = \lambda(\partial_T^\vee(\chi))(f,g) 
\end{multline*}
so the equality $(\partial_A^\vee\circ \alpha(\chi))+(\alpha(\chi)^\vee\circ
\partial_A) = \lambda(\partial_T^\vee(\chi))$ holds, as demanded.
\end{proof}

\vspace{4mm}
\begin{lem}\label{Lem:Derivationsdt0}
Let $(\partial_T,\partial_A)\in \Der_k(P(M),G)$ be a derivation. If
$\partial_A=0$, then $w\circ \partial_T^\vee = 0$. Reciprocally, if
$\partial_T:T_P\to T$ is any morphism such that $w\circ \partial_T^\vee = 0$,
then $(\partial_T,0)$ is a derivation.
\end{lem}

\begin{proof}
By definition of $w$ and $\lambda$, the equality
$$w(\partial_T^\vee(\chi))=\lambda(\partial_T(\chi))(v,v^\vee)$$
holds for all $\chi\in T^\vee$. Hence, if $(\partial_T,0)$ is a derivation, then
we have $\lambda(\partial_T(\chi))=0$ for all $\chi\in T^\vee$ by the Leibniz
rule, therefore $w\circ \partial_T^\vee = 0$. On the other hand, if
$w(\partial_T^\vee(\chi)) = 0$ then we have
$\lambda(\partial_T(\chi))(v,v^\vee)$, and since $(v,v^\vee)$ generates $P_A$,
we get $\lambda(\partial_T^\vee(\chi))=0$.
\end{proof}

\vspace{4mm}
\begin{para}\label{Par:ConstrDerivationPhi}
We will now construct a linear map 
$$\Phi: D_M(k)\otimes\IQ\to \Der_k(P(M),G)\otimes\IQ$$
which will eventually turn out to be an isomorphism, as follows: Given a
deficient point $Q\in D_M(k)$, define a 1--motive $M^+ = [Y\oplus\IZ
\xrightarrow{\:\:u\:\:}G]$ by $u^+(y,n)=u(y)+nQ$. Let
$$r:P(M)\to P(M^+)$$
be a morphism whose composition with the morphism $P(M^+)\to P(M)$ induced by
$M\to M^+$ is multiplication by some nonzero integer $n$. Such a morphism exists
by Corollary \ref{Cor:DeficientGeometric}. We get morphisms
\begin{eqnarray*}
\partial_T & :\:\: & T_{P(M)} \xrightarrow{\:\: r_T\:\:}T_{P(M^+)}
\xrightarrow{\:\: \subseteq \:\:} \cHom(Y\oplus \IZ,T) \xrightarrow{\:\: f
\mapsto f(0,1)\:\:}T\\
\partial_A & :\:\: & A_{P(M)} \xrightarrow{\:\: r_A\:\:}A_{P(M^+)}
\xrightarrow{\:\: \subseteq \:\:} \cHom(Y\oplus \IZ,A)\oplus\cHom(T^\vee,A^\vee)
\xrightarrow{\:\: (f,g) \mapsto f(0,1)\:\:}A
\end{eqnarray*}
and set $\Phi(Q\otimes 1)= (\partial_T,\partial_A)\otimes n^{-1}$. By Lemma
\ref{Lem:TheseAreDerivations} the pair $(\partial_T,\partial_A)$ is indeed a
derivation, and $(\partial_T,\partial_A)\otimes n^{-1}$ does not depend on the
choice of the isogeny $r$, hence $\Phi$ is a well--defined $\IQ$--linear map.
\end{para}

\vspace{4mm}
\begin{thm}\label{Thm:DerivationsMain}
The map $\Phi$ constructed in \ref{Par:ConstrDerivationPhi} is an isomorphism.
\end{thm}

\begin{proof}
Let us write $T_0$ for the torus dual to $\ker w$, where $w:T_P^\vee \to
A_P^\vee$ is the map defined by $P$. So $T_0$ is the largest torus quotient of
$P$ via the canonical projection $\pi_0:P\to T_0$. The torus $T_0$ comes
equipped with a special point $u_0 := \pi(\overline u)$. Every subtorus of $T_0$
that contains $u_0$ is already equal to $T_0$.\\
We now check injectivity and surjectivity of the map $\Phi$, starting with
injectivity. Let $q \in D_M(k)$ be a deficient point such that $\Phi(q\otimes 1)
= 0$. Replacing $M$ by $M^+$, we can suppose without loss of generality that
$q=u(y^+)$ for some $y^+\in Y$. Because $\Der_k(P(M),G)$ is a finitely generated
free group, the relation $\Phi(q\otimes 1) = 0$ means that the maps
$$\partial_A:A_P \xrightarrow{\:\:\subseteq\:\:} A_U
\xrightarrow{\:\:(f,g)\mapsto f(y^+)\:\:}A \qqet \partial_T:T_P
\xrightarrow{\:\:\subseteq\:\:} T_U \xrightarrow{\:\:h\mapsto h(y^+)\:\:}T$$
are both zero. We have $\pi_A(q) = v(y^+) = \partial_A(v,v^\vee) = 0$, so $q$
must be an element of $T(k)$. But then we have $q = u_0(y^+) =
\partial_T(u_0)=0$, using Lemma \ref{Lem:Derivationsdt0}. This shows injectivity
of $\Phi$. \\
To show that $\Phi$ is surjective, let $\partial = (\partial_T,\partial_A)$ be a
derivation, and let us construct a deficient point $q$ with $\Phi(q\otimes 1) =
\partial \otimes 1$. Define $a := \partial_A(v,v^\vee) \in A(k)$ and let
$\widetilde q \in G$ be any point with $\pi(\widetilde q) = a$. Define $M^+ =
[Y\oplus\IZ \xrightarrow{\:\:u^+\:\:}G]$ by $u^+(y,n)=u(y)+n\widetilde q$ and
let
$$\rho:P(M^+) \to P(M)$$
be the induced morphism of semiabelian varieties. The kernel of $\rho$ is
contained in $T_{P(M^+)}$. We get two derivations $\widetilde \partial =
(\widetilde \partial_T,\widetilde \partial_A)$ and $\partial \circ \rho =
(\partial_T\circ\rho_T,\partial_A\circ\rho_A)$ on $P(M^+)$ with values in $G$.
Their difference is the derivation
$$\widetilde \partial - \partial \circ \rho = (\widetilde
\partial_T-\partial_T\circ\rho_T,0)$$
so $\delta := \widetilde \partial_T-\partial_T\circ\rho_T$ is a morphism from
the maximal torus quotient $T_0^+$ of $P(M^+)$ to $T$. The point $q :=
\widetilde q-\delta(\pi(\overline u^+))$ has the required property.
\end{proof}

\vspace{4mm}
\begin{cor}
Let $M=[Y\xrightarrow{\:\:u\:\:}G]$ be a 1--motive over $k$. The group of
deficient points $D_M(k)$ is finitely generated, and its rank is the same as the
rank of $\Der_k(P(M),G)$.
\end{cor}

\begin{proof}
We have an isomorphism $\Phi: D_M(k)\otimes \IQ \to \Der_k(P(M),G) \otimes \IQ$,
so all we have to show is that the group $D_M(k)$ is finitely generated. One can
show (by d\'evissage) that the group $G(k)$ is abstractly isomorphic to a direct
sum of a finite group and a free group. So any subgroup of $G(k)$, in particular
$D_M(k)$, also is isomorphic to a direct sum of a finite group and a free group.
\end{proof}

\vspace{14mm}
\section*{\texorpdfstring{Appendix (following P.~Deligne): The Tannakian
construction of $P(M)$}{Appendix (following P.~Deligne): The Tannakian
construction of P(M)}}
\renewcommand{\thesection}{A}
\setcounter{thm}{0}

\begin{par}
I reproduce here in almost unaltered form a comment of P.~Deligne, explaining
our construction of the Lie algebra of the unipotent motivic fundamental group
$P(M)$ of a 1--motive $M$ in terms of Tannakian formalism. I alone am to blame
for mistakes.
\end{par}

\vspace{4mm}
\begin{par}
The starting point is that in definition \ref{Def:P(M)} the point $\overline u$
should be seen as an extension of $W_{-1}\cEnd(M)$ by the unit object $\IZ$.
This extension can be constructed in a general setting as follows:
\end{par}

\vspace{4mm}
\begin{para}\label{Para:Story1}
Let $\cT$ be a tannakian category in characteristic zero with unit object
$\mathds{1}$. Suppose each object of $\cT$ has a functorial exhaustive
filtration $W$ compatible with tensor products and duals, the functor
$\gr^W_\ast$ being exact. For any object $M$, the object $\cEnd(M) =
M^\vee\otimes M$ contains the subobject $W_{-1}\cEnd(M)$. The filtration of
$\cEnd(M)$ is deduced from the filtration of $M$, hence
$$W_{-1}\cEnd(M) = \im\Big(\bigoplus_p\cHom(M/W_pM,W_pM) \to \cEnd(M)\Big)$$
As $\Ext^1(M/W_pM,W_pM) = \Ext^1(\mathds{1}, \cHom(M/W_pM,W_pM))$, we can take
the class of $M$ in each of the vector spaces $\Ext^1(M/W_pM,W_pM)$, interpret
it as a class in $\Ext^1(\mathds{1}, \cHom(M/W_pM,W_pM))$, take the sum of those
and push to $W_{-1}\cEnd(M)$ by functoriality of $\Ext^1$. We get a class
$$\cl(M) \in \Ext^1(\mathds{1}, W_{-1}\cEnd(M))$$
which is natural in $M$ under taking subobjects and quotients.
\end{para}

\vspace{4mm}
\begin{para}\label{Para:Story2}
\begin{par}
Let $G$ be the fundamental group of $\cT$. It has an invariant unipotent
subgroup $W_{-1}(G)$ with Lie algebra $W_{-1}(\Lie G)$. 
Thus, for all objects $M$, the group $G$ acts on $M$ respecting the filtration
$W$, and the unipotent subgroup $W_{-1}(M)$ acts trivially on $\gr^W_\ast M$.
Let $w:\IG_m\to G/W_{-1}(G)$ be the cocharacter defined by the numbering of the
filtration of $W$ of $M$, for any $M$. This means that for $a\in\IG_m$, the
automorphism $w(a)$ of $\gr^W_\ast(M)$ acts as multiplication with $a^n$ on
$\gr^W_n(M)$. Let us consider the inverse image of this $\IG_m$ in $G$, or more
precisely $\IG_m\times_{G/W_{-1}G}G$. Its Lie algebra $\mathfrak L$ is an
extension of $\Lie(\IG_m)=\mathds{1}$ by $\Lie W_{-1}(G)$. 
\end{par}
\begin{par}
For any object $M$ of $\cT$, let $\fu_M$ be the image in $W_{-1}\cEnd(M)$ of
$\Lie W_{-1}(G)$. The extension $\mathfrak L$ gives us by functoriality a
quotient $\mathfrak L_M$ of $\mathfrak L$ acting on $M$, which is an extension
of $\mathds{1}$ by $\fu_M$.
\end{par}
\end{para}

\vspace{4mm}
\begin{par}
In the setting of section \ref{Sec:UnipoMotFG}, $U(M)$ corresponds to
$W_{-1}\cEnd(M)$, and the subobject $P(M)$ of $U(M)$ corresponds to the
subobject $\fu_M$ of $W_{-1}\cEnd(M)$. We have a 1--motive $[\IZ \to U(M)]$
given by the point $\overline u$, which has to be seen as an extension of $[0\to
U(M)]$ by $[\IZ\to 0]$, and $P(M)$ was defined to be the smallest subobject of
$U(M)$ from which this extension comes, say after passing to
$\IQ$--coefficients. Exactly in this way the two stories \ref{Para:Story1} and
\ref{Para:Story2} are related:
\end{par}

\begin{prop}\label{Pro:Delignes2Stories}
\begin{enumerate}
 \item The extension $\cl(M)$ is the push out of the extension $\mathfrak L_M$
via the inclusion $\fu_M\to W_{-1}\cEnd(M)$.
 \item If $\mathfrak v$ is a subobject of $\fu_M$ such that $\cl(M)$ is the
image of a class in $\Ext^1(\mathds{1}, \mathfrak v)$, then $\mathfrak v =
\fu_M$.
\end{enumerate}
\end{prop}

\begin{par}
I claim in (1) an equality in $\Ext^1(\mathds{1}, W_{-1}\cEnd(M))$. As
$\Hom(\mathds{1}, W_{-1}\cEnd(M))$ is $0$ for weight reasons, it is the same as
an (unique) isomorphism of actual extensions. 
\end{par}

\begin{par}
A first computation, which was helpful for understanding but now has disappeared
from the proof, was to consider the category of graded representations of a
graded Lie algebra $\fg$ with degrees $<0$, and wondering what cocycle $c:\fg 
\to W_{-1} \cEnd(M)$ was giving me $\cl(M)$ in $H^1(\fg, W_{-1}\cEnd(M))$. The
answer is the composite map $\fg \to \fg  \to W_{-1} \cEnd(M)$, where the first
map is multiplication by $n$ in degree $-n$, and the second map is given by the
action of $\fg$ on $M$.
\end{par}

\vspace{4mm}
\begin{proof}[Proof of Proposition \ref{Pro:Delignes2Stories}]
\begin{par}
(1) For each integer $n$ we have a map
$$e_n: \mathds{1} \to \cEnd(\gr^W_n\! M) \to \gr^W_0\cEnd(M) =
\bigoplus_p\cEnd(\gr^W_p\! M)$$
The push--out by $\fu_M \to W_{-1}\cEnd(M)$ of the extension $0\to \fu_M \to
\mathfrak L_M \to \mathds{1} \to 0$ is the pull--back of the extension 
\begin{equation}\label{Eqn:DeligneExtension}
1\to W_{-1}\cEnd(M) \to W_0\cEnd(M) \to \gr^W_0\!\cEnd(M) \to 1\tag{$\ast$}
\end{equation}
by $\sum_nne_n$. Indeed, $\sum_nne_n$ gives the action of $\mathds{1} =
\Lie(\IG_m)$ on $\gr^W_\ast\!M$ corresponding to the grading. Define $E_p :=
\sum_{n>p}e_n$ and let $A\leq B$ be integers such that $W_AM=0$ and  $W_BM=M$.
We have then
$$\sum_nne_n = \sum_{\hspace{-2mm}A\leq p\leq B\hspace{-2mm}}E_p + m\sum_ne_n$$
for some integer $m$ depending on the choice of $A$ and $B$. The map $\sum_ne_n$
lifts to the identity morphism of $M$, viewed as a map $\mathds{1}\to \cEnd(M)$
which factors over $W_0\cEnd(M)$. The pull--back of (\ref{Eqn:DeligneExtension})
by $m\sum_ne_n$ is hence trivial, and the push--out of $\mathfrak L_M$ by the
inclusion $\fu_M \to W_{-1}\cEnd(M)$ is the sum of the pull--backs of
(\ref{Eqn:DeligneExtension}) by the $E_p$. This pull--back by $E_p$ comes from
an extension of $\mathds 1$ by $\cHom(M/W_pM,W_pM)$, which I would like to call
\emph{Endomorphisms of $M$ respecting $W_p(M)$, inducing a multiple of the
identity on $M/W_pM$ and $0$ on $W_pM$}. At least, that is what it becomes in
any realisation. This is the extension already considered in \ref{Para:Story1},
corresponding to the class of $M$ in $\Ext^1(M/W_pM, W_pM)$. The sum of those
extension classes, pushed to $W_{-1}\cEnd(M)$, had been defined to be $\cl(M)$.
\end{par}
\begin{par}
(2) On objects $N$ of weights $< 0$, i.e.\ such that $N = W_{-1}N$, the functor
$\Ext^1(\mathds 1, -)$ is left exact. For any class $\alpha$ in $\Ext^1(\mathds
1, N)$, there is hence a smallest sub-object $N_0$ of $N$ such that $\alpha$
comes from a class in $\Ext^1(\mathds 1, N_0)$. Indeed, if $\alpha$ comes from
$N'$ and from $N''$, the short exact sequence 
$$0\to N'\cap N'' \to N'\oplus N'' \to N$$
shows after applying $\Ext^1(\mathds 1, -)$ that $\alpha$ also comes from
$N'\cap N''$. It remains to show that the class of $\mathfrak L_M$ in
$\Ext^1(\mathds 1, \fu_M)$ does not come from any $\mathfrak v \subsetneq
\fu_M$. Indeed, any subobject of $\mathfrak L_M$ is stable by the action of
$\mathfrak L_M$ because $G$ and $\Lie G$ act on everything.
\end{par}
\end{proof}

\vspace{4mm}
\begin{prop}\label{Pro:DeligneMinimalSubalg}
Any Lie--subobject of $\mathfrak L_M$ mapping onto $\mathds 1$ is equal to
$\mathfrak L_M$. In other words, the extension $0\to \fu_M \to \mathfrak L_M \to
\mathds{1} \to 0$ is essential.
\end{prop}

\begin{proof}
The bracket of $\mathfrak L_M$ passes to the quotients by $W$ and defines
brackets
$$W_p \mathfrak L_M /W_{p-1} \mathfrak L_M \otimes \mathfrak L_M /W_{-1}
\mathfrak L_M \to W_p \mathfrak L_M /W_{p-1}\mathfrak L_M$$
and this bracket $\gr^W_p\mathfrak L_M \otimes\mathds 1 \to \gr^W_p\mathfrak
L_M$ is the multiplication by $p$. For any subobject $\mathfrak L'$ of
$\mathfrak L_M$ mapping onto the quotient $\mathds 1$, the stability of
$\mathfrak L'$ by the action of $W_p \mathfrak L_M$ hence gives a surjectivity
$$\gr^W_p \mathfrak L' \to \gr^W_p \mathfrak L_M$$
from which the equality $\mathfrak L' = \mathfrak L_M$ follows.
\end{proof}

\vspace{14mm}
\bibliographystyle{amsalpha}

\providecommand{\bysame}{\leavevmode\hbox to3em{\hrulefill}\thinspace}

\providecommand{\href}[2]{#2}

\vspace{15mm}
$$ $$
\hspace{70mm}
\begin{minipage}[]{80mm}
Peter Jossen\\[1mm]
Fakult\"at f\"ur Mathematik\\
Universit\"at Regensburg\\
Universit\"atsstr. 31\\
93040 Regensburg, GERMANY\\[2mm]
{\tt peter.jossen@gmail.com} 
\end{minipage}

\end{document}